\documentclass[10pt,reqno]{amsart}

\usepackage{amsthm, mathrsfs, amsmath, amstext, amsxtra, amsfonts, dsfont, amssymb, color}
\usepackage{lmodern}
\definecolor{green_dark}{rgb}{0,0.6,0}
\usepackage[colorlinks, linkcolor=red, citecolor=blue, urlcolor=green_dark, pagebackref, hypertexnames=false]{hyperref}

\newcommand{\N}{\mathbb N}
\newcommand{\Z}{\mathbb Z}

\newcommand{\R}{\mathbb R}
\newcommand{\C}{\mathbb C}

\newcommand{\ep}{\epsilon}

\newcommand{\re}[1]{\mbox{Re} \ #1} 
\newcommand{\im}[1]{\mbox{Im} \ #1} 
\newcommand{\scal}[1]{\left\langle #1 \right\rangle} 

\newcommand{\defendproof}{\hfill $\Box$} 

\newtheorem{theorem}{Theorem}[section]
\newtheorem{defi}[theorem]{Definition}
\newtheorem{lem}[theorem]{Lemma} 
\newtheorem{prop}[theorem]{Proposition}
\newtheorem{coro}[theorem]{Corollary} 
\theoremstyle{definition}
\newtheorem{rem}[theorem]{Remark}

\title[Focusing mass-critical NL4S below energy space]{On the focusing mass-critical nonlinear fourth-order Schr\"odinger equation below the energy space} 

\author[V. D. Dinh]{Van Duong Dinh}
\address[V. D. Dinh]{Institut de Math\'ematiques de Toulouse UMR5219, Universit\'e Toulouse CNRS, 31062 Toulouse Cedex 9, France}
\email{dinhvan.duong@math.univ-toulouse.fr}


\begin{document}

\maketitle
\begin{abstract}
In this paper, we consider the focusing mass-critical nonlinear fourth-order Schr\"odinger equation. We prove that blowup solutions to this equation with initial data in $H^\gamma(\R^d), 5\leq d \leq 7, \frac{56-3d+\sqrt{137d^2+1712d+3136}}{2(2d+32)} <\gamma<2$ concentrate at least the mass of the ground state at the blowup time. This extends the work in \cite{ZhuYangZhang11} where Zhu-Yang-Zhang studied the formation of singularity for the equation with rough initial data in $\R^4$. We also prove that the equation is globally well-posed with initial data $u_0 \in H^\gamma(\R^d), 5\leq d \leq 7, \frac{8d}{3d+8}<\gamma<2$ satisfying $\|u_0\|_{L^2(\R^d)} <\|Q\|_{L^2(\R^d)}$, where $Q$ is the solution to the ground state equation.
\end{abstract}


\section{Introduction}
\setcounter{equation}{0}
Consider the focusing mass-critical nonlinear fourth-order Schr\"odinger equation, namely
\begin{align}
\left\{
\begin{array}{rcl}
i\partial_t u(t,x) + \Delta^2 u(t,x) &=& (|u|^{\frac{8}{d}} u)(t,x), \quad t\geq 0, x\in \R^d, \\
u(0,x) &=& u_0(x) \in H^\gamma(\R^d), 
\end{array}
\right.
\tag{NL4S}
\end{align}
where $u(t,x)$ is a complex valued function in $\R^+ \times \R^d$. The fourth-order Schr\"odinger equation was introduced by Karpman \cite{Karpman} and Karpman-Shagalov \cite{KarpmanShagalov} taking into account the role of small fourth-order dispersion terms in the propagation of intense laser beams in a bulk medium with Kerr nonlinearity. Such a fourth-order Schr\"odinger equation is of the form 
\begin{align}
i\partial_t u + \Delta^2 u +\varepsilon \Delta u + \mu |u|^{\nu-1} u =0, \quad u(0)=u_0, \label{generalized fourth-order schrodinger equation}
\end{align}
where $\varepsilon \in \{ 0, \pm 1\}, \mu \in \{\pm 1\}$ and $\nu>1$. The (NL4S) is a special case of $(\ref{generalized fourth-order schrodinger equation})$ with $\varepsilon =0$ and $\mu=-1$. \newline
\indent The (NL4S) enjoys a natural scaling invariance, that is if $u$ solves (NL4S), then for any $\lambda>0$, 
\begin{align}
u_\lambda(t,x):= \lambda^{-\frac{d}{2}} u(\lambda^{-4} t, \lambda^{-1} x) \label{scaling invariance}
\end{align}
solves the same equation with initial data $u_\lambda(0,x)=\lambda^{-\frac{d}{2}} u_0(\lambda^{-1} x)$. This scaling also preserves the $L^2$-norm, i.e. $\|u_\lambda(0)\|_{L^2(\R^d)} = \|u_0\|_{L^2(\R^d)}$. It is known (see \cite{Dinhfract, Dinhfourt}) that the (NL4S) is locally well-posed in $H^\gamma(\R^d)$ for $\gamma \geq 0$ satisfying for $d\ne 1, 2, 4$,
\begin{align}
\lceil \gamma \rceil \leq 1+\frac{8}{d}, \label{regularity condition locally well-posed}
\end{align} 
where $\lceil \gamma \rceil$ is the smallest integer greater than or equal to $\gamma$. This condition ensures the nonlinearity to have enough regularity. Moreover, the unique solution enjoys mass conservation, i.e.
\[
M(u(t)):=\|u(t)\|^2_{L^2(\R^d)} = \|u_0\|^2_{L^2(\R^d)},
\] 
and $H^2$-solution has conserved energy, i.e.
\[
E(u(t)):=\int_{\R^d} \frac{1}{2}|\Delta u(t,x)|^2 - \frac{d}{2d+8}|u(t,x)|^{\frac{2d+8}{d}} dx = E(u_0).
\]
In the sub-critical regime, i.e. $\gamma>0$, the time of existence depends only on the $H^\gamma$-norm of the initial data. Let $T^*$ be the maximal time of existence. The local well-posedness gives the following blowup alternative criterion: either $T^*=\infty$ or 
\[
T^*<\infty, \quad  \lim_{t\rightarrow T^*} \|u(t)\|_{H^\gamma(\R^d)} = \infty.
\]
The study of blowup solutions for the focusing nonlinear fourth-order Schr\"odinger equation has been attracted a lot of interest in a past decay (see e.g. \cite{FibichIlanPapanicolaou}, \cite{BaruchFibichMandelbaum}, \cite{ZhuYangZhang10}, \cite{ZhuYangZhang11}, \cite{BoulengerLenzmann} and references therein). It is closely related to ground states $Q$ of (NL4S) which are solutions to 
the elliptic equation
\begin{align}
\Delta^2 Q(x)-Q(x) + |Q(x)|^{\frac{8}{d}} Q(x)=0. \label{ground state equation}
\end{align} 
The equation $(\ref{ground state equation})$ is obtained by considering solitary solutions (standing waves) of (NL4S) of the form $u(t,x)=Q(x)e^{-it}$. The existence of solutions to $(\ref{ground state equation})$ is proved in \cite{ZhuYangZhang10}, but the uniqueness of the solution is still an open problem. In the case $\|u_0\|_{L^2(\R^d)}<\|Q\|_{L^2(\R^d)}$, using the sharp  Gagliardo-Nirenberg inequality (see \cite{FibichIlanPapanicolaou} or \cite{ZhuYangZhang10}), namely
\begin{align}
\|u\|^{2+\frac{8}{d}}_{L^{2+\frac{8}{d}}(\R^d)} \leq C(d) \|u\|^{\frac{8}{d}}_{L^2(\R^d)} \|\Delta u\|^2_{L^2(\R^d)}, \quad C(d):=\frac{1+\frac{4}{d}}{\|Q\|^{\frac{8}{d}}_{L^2(\R^d)}}, \label{sharp gargliardo nirenberg inequality}
\end{align}
together with the energy conservation, Fibich-Ilan-Papanicolaou in \cite{FibichIlanPapanicolaou} (see also \cite{BaruchFibichMandelbaum}) proved that the (NL4S) is globally well-posed in $H^2(\R^d)$. Moreover, the authors in \cite{FibichIlanPapanicolaou} also provided some numerical observations showing that the $H^2$-solution to (NL4S) may blowup if the initial data satisfies $\|u_0\|_{L^2(\R^d)} \geq \|Q\|_{L^2(\R^d)}$. Baruch-Fibich-Mandelbaum in \cite{BaruchFibichMandelbaum} proved some dynamical properties of the radially symmetric blowup solution such as blowup rate, $L^2$-concentration. Later, Zhu-Yang-Zhang in \cite{ZhuYangZhang10} removed the radially symmetric assumption and established the profile decomposition, the existence of the ground state of elliptic equation $(\ref{ground state equation})$ and the following concentration compactness property for the (NL4S). 
\begin{theorem}[Concentration compactness \cite{ZhuYangZhang10}] \label{theorem concentration compactness}
Let $(v_n)_{n\geq 1}$ be a bounded family of $H^2(\R^d)$ functions such that 
\[
\limsup_{n\rightarrow \infty} \|\Delta v_n\|_{L^2(\R^d)} \leq M <\infty \quad \text{and} \quad \limsup_{n\rightarrow \infty} \|v_n\|_{L^{2+\frac{8}{d}}(\R^d)} \geq m >0.
\] 
Then there exists a sequence $(x_n)_{n\geq 1}$ of $\R^d$ such that up to a subsequence
\[
v_n(\cdot+x_n) \rightharpoonup V \text{ weakly in } H^2(\R^d) \text{ as } n\rightarrow \infty,
\]
with $\|V\|^{\frac{8}{d}}_{L^2(\R^d)} \geq \frac{\|Q\|^{\frac{8}{d}}_{L^2(\R^d)} m^{2+\frac{8}{d}}}{\left(1+\frac{4}{d}\right) M^2}$, where $Q$ is the solution to the ground state equation $(\ref{ground state equation})$. 
\end{theorem}
Consequently, the authors in \cite{ZhuYangZhang11} established the limiting profile and $L^2$-concentration for (NL4S) with initial data $u_0 \in H^\gamma(\R^4), \frac{9+\sqrt{721}}{20}<\gamma<2$. Recently, Boulenger-Lenzmann in \cite{BoulengerLenzmann} proved a general result on finite-time blowup for the focusing generalized nonlinear fourth-order Schr\"odigner equation( i.e. $(\ref{generalized fourth-order schrodinger equation})$ with $\mu=1$) with radial data in $H^2(\R^d)$.  \newline 
\indent The goal of this paper is to extend the results of \cite{ZhuYangZhang11} to higher dimensions $d\geq 5$ and to prove the global existence of (NL4S) for initial data $u_0\in H^\gamma(\R^d), 0<\gamma<2$ satisfying $\|u_0\|_{L^2(\R^d)}<\|Q\|_{L^2(\R^d)}$. Since we are working with low regularity data, the energy  argument does not work. In order to overcome this problem, we make use of the $I$-method. Due to the high-order term $\Delta^2 u$, we requires the nonlinearity to have at least two orders  of derivatives in order to successfully establish the almost conservation law. We thus restrict ourself in spatial space of dimensions $d=5, 6, 7$. Our main results are as follows. 
\begin{theorem} \label{theorem weak limiting profile}
Let $d= 5, 6, 7$ and $u_0 \in H^\gamma(\R^d)$ with $\frac{56-3d+\sqrt{137d^2+1712d+3136}}{2(2d+32)} <\gamma<2$. If the corresponding solution to the \emph{(NL4S)} blows up in finite time $0<T^*<\infty$, then there exists a function $U \in H^2(\R^d)$ such that $\|U\|_{L^2(\R^d)} \geq \|Q\|_{L^2(\R^d)}$ and there exist sequences $(t_n, \lambda_n, x_n)_{n\geq 1} \in \R^+ \times \R^+_* \times \R^d$ satisfying 
\[
t_n \nearrow T^* \text{ as } n \rightarrow \infty \quad \text{ and } \quad \lambda_n \lesssim (T^*-t_n)^{\frac{\gamma}{8}}, \quad \forall n \geq 1
\]
such that
\[
\lambda_n^{\frac{d}{2}} u(t_n, \lambda_n \cdot +x_n) \rightharpoonup U \text{ weakly in } H^{a(d,\gamma)-}(\R^d) \text{ as } n \rightarrow \infty,
\]
where 
\[
a(d,\gamma):=\frac{4d\gamma^2+(2d+48)\gamma+16d}{16d+(56-3d)\gamma-16\gamma^2},
\] 
and $Q$ is the solution of the ground state equation $(\ref{ground state equation})$.
\end{theorem}  
The proof of the above theorem is based on the combination of the $I$-method and the concentration compactness property given in Theorem $\ref{theorem concentration compactness}$ which is similar to those given in \cite{VisanZhang07} and \cite{ZhuYangZhang11}. The $I$-method was first introduced by $I$-Team in \cite{I-teamalmost} in order to treat the nonlinear Schr\"odinger equation at low regularity. It then becomes a useful way to address the low regularity problem for the nonlinear dispersive equations. The idea is to replace the non-conserved energy $E(u)$ when $\gamma<2$ by an ``almost conserved'' variance $E(Iu)$ with $I$ a smoothing operator which is the identity at low frequency and behaves like a fractional integral operator of order $2-\gamma$ at high frequency. Since $Iu$ is not a solution of (NL4S), we may expect an energy increment. The key is to show that on intervals of local well-posedness, the modified energy $E(Iu)$ is an ``almost conserved'' quantity and grows much slower than the modified kinetic energy $\|\Delta Iu\|^2_{L^2(\R^d)}$. To do so, we need delicate estimates on the commutator between the $I$-operator and the nonlinearity. 
Note that when $d=4$, the nonlinearity is algebraic, one can use the Fourier transform technique to write the commutator explicitly and then control it by multi-linear analysis. In our setting, the nonlinearity is not algebraic. Thus we can not apply the Fourier transform technique. Fortunately, thanks to a special Strichartz estimate $(\ref{strichartz estimate biharmonic fourth-order})$, we are able to apply the technique given in \cite{VisanZhang07} to control the commutator. The concentration compactness property given in Theorem $\ref{theorem concentration compactness}$ is very useful to study the dynamical properties of blowup solutions for the nonlinear fourth-order Schr\"odinger equation. With the help of this property, Zhu-Yang-Zhang proved in \cite{ZhuYangZhang10} the $L^2$-concentration of blowup solutions and the limiting profile of minimal-mass blowup solutions with non-radial data in $H^2(\R^d)$. In \cite{ZhuYangZhang11}, they extended these results for non-radial data below the energy space in the fourth dimensional space. \newline  
\indent As a consequence of Theorem $\ref{theorem weak limiting profile}$, we have the following mass concentration property.
\begin{theorem} \label{theorem mass concentration}
Let $d= 5, 6, 7$ and $u_0 \in H^\gamma(\R^d)$ with $\frac{56-3d+\sqrt{137d^2+1712d+3136}}{2(2d+32)}<\gamma<2$. Assume that the corresponding solution $u$ to the \emph{(NL4S)} blows up in finite time $0<T^*<\infty$. If $\alpha(t)>0$ is an arbitrary function such that 
\[
\lim_{t\nearrow T^*} \frac{(T^*-t)^{\frac{\gamma}{8}}}{\alpha(t)} = 0,
\]
then there exists a function $x(t) \in \R^d$ such that
\[
\limsup_{t\nearrow T^*} \int_{|x-x(t)|\leq \alpha(t)} |u(t,x)|^2 dx \geq \int_{\R^d} |Q(x)|^2dx,
\]
where $Q$ is the solution to the ground state equation $(\ref{ground state equation})$.
\end{theorem}
When the mass of the initial data equals to the mass of the solution of the ground state equation $(\ref{ground state equation})$, we have the following improvement of Theorem $\ref{theorem weak limiting profile}$. Note that in the below result, we assume that there exists a unique solution to the ground state equation $(\ref{ground state equation})$ which is a delicate open problem. 
\begin{theorem} \label{theorem strongly limiting profile}
Let $d= 5, 6, 7$ and $u_0 \in H^\gamma(\R^d)$ with $\frac{56-3d+\sqrt{137d^2+1712d+3136}}{2(2d+32)}<\gamma<2$ be such that $\|u_0\|_{L^2(\R^d)} = \|Q\|_{L^2(\R^d)}$. If the corresponding solution $u$ to the \emph{(NL4S)} blows up in finite time $0<T^*<\infty$, then there exist sequences $(t_n, e^{i\theta_n}, \lambda_n, x_n)_{n\geq 1} \in \R^+\times \mathbb{S}^1\times \R^+_* \times \R^d$ satisfying 
\[
t_n \nearrow T^* \text{ as } n \rightarrow \infty \quad \text{ and } \quad \lambda_n \lesssim (T^*-t_n)^{\frac{\gamma}{8}}, \quad \forall n \geq 1
\]
such that
\[
\lambda_n^{\frac{d}{2}} e^{i\theta_n} u(t_n, \lambda_n \cdot + x_n) \rightarrow Q \text{ strongly in } H^{a(d,\gamma)-}(\R^d) \text{ as } n \rightarrow \infty,
\]
where 
\[
a(d,\gamma):=\frac{4d\gamma^2+(2d+48)\gamma+16d}{16d+(56-3d)\gamma-16\gamma^2},
\]
and $Q$ is the unique solution to the ground state equation $(\ref{ground state equation})$. 
\end{theorem}
Our last result concerns with the global existence of (NL4S) with rough initial data $u_0$ satisfying $\|u_0\|_{L^2(\R^d)}<\|Q\|_{L^2(\R^d)}$. 
\begin{theorem} \label{theorem global existence below ground state}
Let $d=5, 6, 7$ and $u_0\in H^\gamma(\R^d)$ with $\frac{8d}{3d+8}<\gamma<2$ be such that $\|u_0\|_{L^2(\R^d)}<\|Q\|_{L^2(\R^d)}$, where $Q$ is the solution to the ground state equation $(\ref{ground state equation})$. Then the initial value problem \emph{(NL4S)} is globally well-posed. 
\end{theorem}
The proof of this result is inspired by the argument of \cite{FangZhong} which relies on the $I$-method and the sharp Gagliardo-Nirenberg inequality $(\ref{sharp gargliardo nirenberg inequality})$. Using the smallness assumption of the initial data, the sharp Gagliardo-Nirenberg inquality shows that the modified kinetic energy is controlled by the total energy. This allows us to establish the almost conservation law for the modified energy. \newline 
\indent This paper is organized as follows. In Section $\ref{section preliminaries}$, we introduce some notations and recall some results related to our problem. In Section $\ref{section modified local well-posedness}$, we recall some local existence results and prove the modified local well-posedness. In Section $\ref{section modified energy increment}$, we prove two types of modified energy increment. In Section $\ref{section limiting profile}$, we give the proof of Theorem $\ref{theorem weak limiting profile}$, Theorem $\ref{theorem mass concentration}$ and Theorem $\ref{theorem strongly limiting profile}$. Finally, we prove the global well-posedness with small initial data in Section $\ref{section global well-posedness}$.
\section{Preliminaries} \label{section preliminaries}
\setcounter{equation}{0}
In the sequel, the notation $A \lesssim B$ denotes an estimate of the form $A\leq CB$ for some constant $C>0$. The notation $A \sim B$ means that $A \lesssim B$ and $B \lesssim A$. We write $A \ll B$ if $A \leq cB$ for some small constant $c>0$. We also use $\scal{a}:=1+|a|$ and $a\pm:=a\pm \ep$ for some universal constant $0<\ep \ll 1$ and  
\subsection{Nonlinearity}
Let $F(z):= |z|^{\frac{8}{d}} z, d=5, 6, 7$  be the function that defines the nonlinearity in (NL4S). The derivative $F'(z)$ is defined as a real-linear operator acting on $w \in \C$ by
\[
F'(z)\cdot w:= w \partial_z F(z) + \overline{w} \partial_{\overline{z}} F(z), 
\]
where
\[
\partial_z F(z)=\frac{2d+8}{2d} |z|^{\frac{8}{d}}, \quad \partial_{\overline{z}} F(z) = \frac{4}{d}|z|^{\frac{8}{d}} \frac{z}{\overline{z}}.
\]
We shall identify $F'(z)$ with the pair $(\partial_z F(z), \partial_{\overline{z}} F(z))$, and define its norm by
\[
|F'(z)| := |\partial_zF(z)| + |\partial_{\overline{z}} F(z)|.
\]
It is clear that $|F'(z)| = O(|z|^{\frac{8}{d}})$. 
We also have the following chain rule
\[
\partial_k F(u) = F'(u) \partial_k u, 
\] 
for $k \in \{1,\cdots, d\}$. In particular, we have
\[
\nabla F(u)= F'(u) \nabla u. 
\]
\indent We next recall the fractional chain rule to estimate the nonlinearity.
\begin{lem}[Fractional chain rule for $C^1$ functions \cite{ChristWeinstein}, \cite{KenigPonceVega}]\label{lem fractional chain}
Suppose that $G\in C^1(\C, \C)$, and $\alpha \in (0,1)$. Then for $1 <q \leq q_2 <\infty$ and $1<q_1 \leq \infty$ satisfying $\frac{1}{q}=\frac{1}{q_1}+\frac{1}{q_2}$, 
\[
\||\nabla|^\alpha G(u) \|_{L^q_x} \lesssim \|G'(u) \|_{L^{q_1}_x} \||\nabla|^\alpha u \|_{L^{q_2}_x}.
\]
\end{lem}
We refer the reader to \cite[Proposition 3.1]{ChristWeinstein} for the proof of the above estimate when $1<q_1<\infty$, and to \cite[Theorem A.6]{KenigPonceVega} for the proof when $q_1=\infty$. When $G$ is no longer $C^1$, but H\"older continuous, we have the following fractional chain rule.
\begin{lem}[Fractional chain rule for $C^{0,\beta}$ functions \cite{Visanthesis}] \label{lem fractional chain rule holder}
Suppose that $G\in C^{0,\beta}(\C, \C), \beta \in (0,1)$. Then for every $0<\alpha <\beta, 1<q<\infty$, and $\frac{\alpha}{\beta}<\rho<1$,
\[
\||\nabla|^\alpha G(u)\|_{L^q_x} \lesssim \| |u|^{\beta-\frac{\alpha}{\rho}} \|_{L^{q_1}_x} \||\nabla|^\rho u\|^{\frac{\alpha}{\rho}}_{L^{\frac{\alpha}{\rho}q_2}_x},
\] 
provided $\frac{1}{q}=\frac{1}{q_1}+\frac{1}{q_2}$ and $\left(1-\frac{\alpha}{\beta \rho}\right)q_1>1$.
\end{lem}
We refer the reader to \cite[Proposition A.1]{Visanthesis} for the proof of this result. We also need the following fractional Leibniz rule.
\begin{lem}[Fractional Leibniz rule \cite{Kato95}] \label{lem fractional leibniz rule}
Let $F\in C^k(\C, \C), k \in \N\backslash \{0\}$. Assume that there is $\nu\geq k$ such that
\[
|D^iF(z)|\lesssim |z|^{\nu-i}, \quad \forall z\in \C, i=1,...,k.
\]
Then for $\gamma \in [0,k], 1<q \leq q_2<\infty$ and $1<q_1 \leq \infty$ satisfying $\frac{1}{q}=\frac{\nu-1}{q_1}+\frac{1}{q_2}$,
\begin{align}
\||\nabla|^\gamma F(u)\|_{L^q_x} \lesssim \|u\|^{\nu-1}_{L^{q_1}_x} \||\nabla|^\gamma u\|_{L^{q_2}_x}. \label{fractional leibniz rule}
\end{align}
Moreover, if $F$ is a homogeneous polynomial in $u$ and $\overline{u}$, then $(\ref{fractional leibniz rule})$ holds true for any $\gamma\geq 0$.
\end{lem}
The reader can find the proof of this fractional Leibniz rule in \cite[Appendix]{Kato95}.
\subsection{Strichartz estimates} \label{subsection fractional derivative}
Let $I \subset \R$ and $p, q \in [1,\infty]$. We define the mixed norm
\[
\|u\|_{L^p_t(I, L^q_x)} := \Big( \int_I \Big( \int_{\R^d} |u(t,x)|^q dx \Big)^{\frac{1}{q}} \Big)^{\frac{1}{p}}
\] 
with a usual modification when either $p$ or $q$ are infinity. When there is no risk of confusion, we may write $L^p_t L^q_x$ instead of $L^p_t(I,L^q_x)$. We also use $L^p_{t,x}$ when $p=q$.
\begin{defi}
A pair $(p,q)$ is said to be \textbf{Schr\"odinger admissible}, for short $(p,q) \in S$, if 
\[
(p,q) \in [2,\infty]^2, \quad (p,q,d) \ne (2,\infty,2), \quad \frac{2}{p}+\frac{d}{q} \leq \frac{d}{2}.
\]
\end{defi}
Throughout this paper, we denote for $(p,q)\in [1,\infty]^2$,
\begin{align}
\gamma_{p,q}=\frac{d}{2}-\frac{d}{q}-\frac{4}{p}. \label{define gamma pq}
\end{align}
\begin{defi}
A pair $(p,q)$ is called \textbf{biharmonic admissible}, for short $(p,q)\in B$, if 
\[
(p,q) \in S,\quad q<\infty, \quad \gamma_{p,q}=0.
\]
\end{defi}
\begin{prop}[Strichartz estimate for fourth-order Schr\"odinger equation \cite{Dinhfract}] \label{prop generalized strichartz estimate}
Let $\gamma \in \R$ and $u$ be a (weak) solution to the linear fourth-order Schr\"odinger equation namely
\[
u(t)= e^{it\Delta^2}u_0 + \int_0^t e^{i(t-s)\Delta^2} F(s) ds,
\]
for some data $u_0, F$. Then for all $(p,q)$ and $(a,b)$ Schr\"odinger admissible with $q<\infty$ and $b<\infty$,
\begin{align}
\||\nabla|^\gamma u\|_{L^p_t(\R, L^q_x)} \lesssim \||\nabla|^{\gamma+\gamma_{p,q}} u_0\|_{L^2_x} + \||\nabla|^{\gamma+\gamma_{p,q}-\gamma_{a',b'} -4} F\|_{L^{a'}_t(\R, L^{b'}_x)}. \label{generalized strichartz estimate}
\end{align}
Here $(a,a')$ and $(b,b')$ are conjugate pairs, and $\gamma_{p,q}, \gamma_{a',b'}$ are defined as in $(\ref{define gamma pq})$.
\end{prop}
We refer the reader to \cite[Proposition 2.1]{Dinhfract} for the proof of Proposition $\ref{prop generalized strichartz estimate}$. The proof is based on the scaling technique instead of using a dedicate dispersive estimate of \cite{Ben-ArtziKochSaut} for the fundamental solution of the homogeneous fourth-order Schr\"odinger equation. Note that the estimate $(\ref{generalized strichartz estimate})$ is exactly the one given in \cite{MiaoZhang}, \cite{Pausader} or \cite{Pausadercubic} where the author considered $(p,q)$ and $(a,b)$ are either sharp Schr\"odinger admissible, i.e.
\[
p, q \in [2,\infty]^2, \quad (p,q,d) \ne (2,\infty,2), \quad \frac{2}{p}+\frac{d}{q}=\frac{d}{2},
\]
or biharmonic admissible. \newline
\indent The following result is a direct consequence of $(\ref{generalized strichartz estimate})$. 
\begin{coro} \label{coro fourth order strichartz estimate}
Let $\gamma \in \R$ and $u$ a (weak) solution to the linear fourth-order Schr\"odinger equation for some data $u_0, F$. Then for all $(p,q)$ and $(a,b)$ biharmonic admissible,
\begin{align}
\||\nabla|^\gamma u\|_{L^p_t(\R,L^q_x)} \lesssim \|\|\nabla|^\gamma u_0\|_{L^2_x} + \||\nabla|^\gamma F\|_{L^{a'}_t(\R,L^{b'}_x)}, \label{strichartz estimate biharmonic}
\end{align}
and
\begin{align}
\|\Delta u\|_{L^p_t(\R,L^q_x)} \lesssim \|\Delta u_0\|_{L^2_x} + \|\nabla F\|_{L^2_t(\R,L^{\frac{2d}{d+2}}_x)}. \label{strichartz estimate biharmonic fourth-order}
\end{align}
\end{coro}
\subsection{Littlewood-Paley decomposition}
Let $\varphi$ be a radial smooth bump function supported in the ball $|\xi|\leq 2$ and  equal to 1 on the ball $|\xi|\leq 1$.  For $M=2^k, k \in \Z$, we define the Littlewood-Paley operators 
\begin{align*}
\widehat{P_{\leq M} f}(\xi) &:= \varphi(M^{-1}\xi) \hat{f}(\xi), \\
\widehat{P_{>M} f}(\xi) &:= (1-\varphi(M^{-1}\xi)) \hat{f}(\xi), \\
\widehat{P_M f}(\xi) &:= (\varphi(M^{-1} \xi) - \varphi(2M^{-1}\xi)) \hat{f}(\xi),
\end{align*}
where $\hat{\cdot}$ is the spatial Fourier transform. Similarly, we can define
\[
P_{<M} := P_{\leq M}-P_M, \quad P_{\geq M} := P_{>M}+ P_M,
\] 
and for $M_1 \leq M_2$,
\[
P_{M_1 < \cdot \leq M_2}:= P_{\leq M_2} - P_{\leq M_1} = \sum_{M_1 < M \leq M_2} P_M.
\]
We recall the following standard Bernstein inequalities (see e.g. \cite[Chapter 2]{BCDfourier} or \cite[Appendix]{Tao}).
\begin{lem}[Bernstein inequalities] \label{lem bernstein inequalities}
Let $\gamma\geq 0$ and $1 \leq p \leq q \leq \infty$. We have
\begin{align*}
\|P_{\geq M} f\|_{L^p_x} &\lesssim M^{-\gamma} \||\nabla|^\gamma P_{\geq M} f\|_{L^p_x}, \\
\|P_{\leq M} |\nabla|^\gamma f\|_{L^p_x} &\lesssim M^\gamma \|P_{\leq M} f\|_{L^p_x}, \\
\|P_M |\nabla|^{\pm \gamma} f\|_{L^p_x} & \sim M^{\pm \gamma} \|P_M f\|_{L^p_x}, \\
\|P_{\leq M} f\|_{L^q_x} &\lesssim M^{\frac{d}{p}-\frac{d}{q}} \|P_{\leq M} f\|_{L^p_x}, \\
\|P_M f\|_{L^q_x} &\lesssim M^{\frac{d}{p}-\frac{d}{q}} \|P_M f\|_{L^p_x}.
\end{align*}
\end{lem}
\subsection{$I$-operator}
Let $0\leq \gamma <2$ and $N\gg 1$. We define the Fourier multiplier $I_N$ by
\[
\widehat{I_N f}(\xi):= m_N(\xi) \hat{f}(\xi),
\]
where $m_N$ is a smooth, radially symmetric, non-increasing function such that 
\begin{align*}
m_N(\xi) := \left\{ 
\begin{array}{cl}
1 &\text{if } |\xi|\leq N, \\
(N^{-1}|\xi|)^{\gamma-2} & \text{if } |\xi| \geq 2N.
\end{array}
\right. 
\end{align*}
We shall drop the $N$ from the notation and write $I$ and $m$ instead of $I_N$ and $m_N$. We recall (see \cite[Lemma 2.7]{Dinhmass}) some basic properties of the $I$-operator in the following lemma. 
\begin{lem} \label{lem properties I operator}
Let $0\leq \sigma \leq \gamma<2$ and $1<q<\infty$. Then
\begin{align}
\|I f\|_{L^q_x} &\lesssim \|f\|_{L^q_x}, \label{property 1} \\
\| |\nabla|^\sigma P_{>N} f\|_{L^q_x} &\lesssim N^{\sigma-2} \|\Delta I f\|_{L^q_x}, \label{property 2} \\
\|\scal{\nabla}^\sigma f\|_{L^q_x} &\lesssim \|\scal{\Delta} I f\|_{L^q_x}, \label{property 3} \\
\|f\|_{H^\gamma_x} \lesssim \|If\|_{H^2_x} &\lesssim N^{2-\gamma} \|f\|_{H^\gamma_x}, \label{property 4} \\
\|If\|_{\dot{H}^2_x} &\lesssim N^{2-\gamma} \|f\|_{\dot{H}^\gamma_x}. \label{property 5}
\end{align}
\end{lem}
When the nonlinearity $F(u)$ is algebraic, one can use the Fourier transform to write the commutator like $F(Iu)-IF(u)$ as a product of Fourier transforms of $u$ and $Iu$, and then measure the frequency interactions. However, in our setting, the nonlinearity is no longer algebraic, we thus need the following rougher estimate which is a modified version of the Schr\"odinger context (see \cite{VisanZhang07}). 
\begin{lem}\label{lem rougher estimate}
Let $1<\gamma<2, 0<\delta <\gamma-1$ and $1<q, q_1, q_2 <\infty$ be such that $\frac{1}{q}=\frac{1}{q_1}+\frac{1}{q_2}$. Then
\begin{align}
\|I(fg)-(If)g\|_{L^q_x} \lesssim N^{-(2-\gamma+\delta)} \|If\|_{L^{q_1}_x} \|\scal{\nabla}^{2-\gamma+\delta} g\|_{L^{q_2}_x}.
\end{align}
\end{lem}  
We refer the reader to \cite[Lemma 2.9]{Dinhmass} for the proof of this result. A direct consequence of Lemma $\ref{lem rougher estimate}$ with the fact that
\[
\nabla F(u) = \nabla u F'(u)
\]
is the following commutator estimate.
\begin{coro} \label{coro rougher estimate}
Let $1<\gamma<2, 0<\delta<\gamma-1$ and $1<q, q_1, q_2<\infty$ be such that $\frac{1}{q}=\frac{1}{q_1}+\frac{1}{q_2}$. Then 
\begin{align}
\|\nabla I F(u)-(I\nabla u)F'(u)\|_{L^q_x} \lesssim N^{-(2-\gamma+\delta)} \|\nabla I u\|_{L^{q_1}_x} \|\scal{\nabla}^{2-\gamma+\delta} F'(u)\|_{L^{q_2}_x}. \label{rougher estimate} 
\end{align}
\end{coro}
\section{Modified local well-posedness} \label{section modified local well-posedness}
\setcounter{equation}{0}
We firstly recall the local theory for (NL4S) in Sobolev spaces (see \cite{Dinhfract, Dinhfourt}). 
\begin{prop}[Local well-posedness in Sobolev spaces] \label{prop local well-posedness}
Let $5\leq d\leq 7, 0<\gamma<2$ and $u_0\in H^\gamma(\R^d)$. Then the equation \emph{(NL4S)} is locally well-posed on $[0,T_{\emph{lwp}}]$ with 
\[
T_{\emph{lwp}} \sim \|u_0\|^{-\frac{4}{\gamma}}_{H^\gamma_x}.
\]
Moreover, 
\[
\sup_{(a,b)\in B} \|u\|_{L^a_t([0,T_{\emph{lwp}}], W^{\gamma,b}_x)} \lesssim \|u_0\|_{H^\gamma_x}.
\]
The implicit constants depend only on the dimension $d$ and the regularity $\gamma$.
\end{prop}
\begin{proof}
Let us introduce
\[
p=\frac{2(d+4)}{d-2\gamma}, \quad q=\frac{2d(d+4)}{d^2+8\gamma}.
\]
It is easy to check that $(p,q)$ is biharmonic admissible. We next choose $(m,n)$ so that
\begin{align}
\frac{1}{p'}=\frac{\frac{8}{d}}{m} +\frac{1}{p}, \quad \frac{1}{q'}=\frac{\frac{8}{d}}{n} +\frac{1}{q}, \label{define m n}
\end{align}
or
\[
m=\frac{4(d+4)}{d(2+\gamma)}, \quad n=\frac{2(d+4)}{d-2\gamma}.
\]
With this choice of $n$, we have the Sobolev embedding $\dot{W}^{\gamma,q}_x \hookrightarrow L^n_x$. \newline
\indent Now, we consider
\[
X:=\Big\{ u \in L^p_t([0,T], W^{\gamma, q}_x) \ | \ \|u\|_{L^p_t([0,T], W^{\gamma,q}_x)} \leq M  \Big\}
\]
equipped with the distance
\[
d(u,v):=\|u-v\|_{L^p_t([0,T], L^q_x)},
\]
where $T, M>0$ to be chosen later. By Duhamel's formula, it suffices to prove that the functional 
\[
\Phi(u)(t):=e^{it\Delta^2} u_0 - i\int_0^t e^{i(t-s)\Delta^2} |u(s)|^{\nu-1} u(s) ds
\]
is a contraction on $(X,d)$. By Strichartz estimate $(\ref{strichartz estimate biharmonic})$,
\begin{align*}
\|\Phi(u)\|_{L^p_t([0,T], W^{\gamma, q}_x)} &\lesssim \|u_0\|_{H^\gamma_x} + \|F(u)\|_{L^{p'}_t([0,T], W^{\gamma, q'}_x)}, \\
\|\Phi(u)-\Phi(v)\|_{L^p_t([0,T], L^q_x)} &\lesssim \|F(u)-F(v)\|_{L^{p'}_t([0,T], L^{q'}_x)},
\end{align*}
where $F(u)=|u|^{\frac{8}{d}} u$ and similarly for $F(v)$. Using $(\ref{define m n})$, we apply Lemma $\ref{lem fractional leibniz rule}$ with $k=2, \gamma\in (0,2), \nu=1+\frac{8}{d}$ to have
\begin{align*}
\|F(u)\|_{W^{\gamma,q'}_x} \lesssim \|u\|^{\frac{8}{d}}_{L^n_x} \|u\|_{W^{\gamma,q}_x} \lesssim \|u\|^{\frac{8}{d}}_{\dot{W}^{\gamma,q}_x} \|u \|_{W^{\gamma,q}_x}.
\end{align*}
Note that $\nu\geq k$ since $5\leq d\leq 7$. Using again $(\ref{define m n})$, the H\"older inequality  and Sobolev embedding then imply
\[
\|F(u)\|_{L^{p'}_t([0,T], W^{\gamma,q'}_x)} \lesssim \|u \|^{\frac{8}{d}}_{L^m_t([0,T],\dot{W}^{\gamma,q}_x)} \|u\|_{L^p_t([0,T], W^{\gamma,q}_x)} 
\lesssim T^{\frac{2\gamma}{d}}\|u\|^{1+\frac{8}{d}}_{L^p_t([0,T], W^{\gamma,q}_x)}.
\]
Similarly, we have
\begin{align*}
\|F(u)-F(v)\|_{L^{p'}_t([0,T],L^{q'}_x)} &\lesssim T^{\frac{2\gamma}{d}} \Big(\|u\|^{\frac{8}{d}}_{L^p_t([0,T],\dot{W}^{\gamma,q}_x}+ \|v\|^{\frac{8}{d}}_{L^p_t([0,T],\dot{W}^{\gamma,q}_x}\Big)\|u-v\|_{L^p_t([0,T],L^q_x)} \\
&\lesssim T^{\frac{2\gamma}{d}} \Big(\|u\|^{\frac{8}{d}}_{L^p_t([0,T],W^{\gamma,q}_x}+ \|v\|^{\frac{8}{d}}_{L^p_t([0,T],W^{\gamma,q}_x}\Big)\|u-v\|_{L^p_t([0,T],L^q_x)}.
\end{align*}
This shows that for all $u, v \in X$, there exists $C>0$ independent of $T$ and $u_0 \in H^\gamma_x$ so that
\begin{align}
\|\Phi(u)\|_{L^p_t([0,T], W^{\gamma,q}_x)} &\lesssim C\|u_0\|_{H^\gamma_x} + C T^{\frac{2\gamma}{d}} M^{1+\frac{8}{d}}, \label{contraction estimate}\\
d(\Phi(u), \Phi(v)) &\lesssim CT^{\frac{2\gamma}{d}} M^{\frac{8}{d}} d(u,v). \nonumber
\end{align}
If we set $M=2C \|u_0\|_{H^\gamma_x}$ and choose 
\[
T\sim \|u_0\|_{H^\gamma_x}^{-\frac{4}{\gamma}},
\]
then $X$ is stable by $\Phi$ and $\Phi$ is a contraction on $(X,d)$. The fixed point argument proves the local existence. Moreover, by Strichartz estimate $(\ref{strichartz estimate biharmonic})$,
\[
\sup_{(a,b)\in B} \|u\|_{L^a_t([0,T], W^{\gamma,b}_x)} \lesssim \|u_0\|_{H^\gamma_x} +\|F(u)\|_{L^{p'}_t([0,T], W^{\gamma,q'}_x)} \lesssim \|u_0\|_{H^\gamma_x}.
\]
The proof is complete.
\end{proof}
\begin{coro}[Blowup criterion] \label{coro blowup criterion}
Let $5\leq d\leq 7, 0 <\gamma<2$ and $u_0\in H^\gamma(\R^d)$. Assume that the unique solution $u$ to \emph{(NL4S)} blows up at time $0<T^*<\infty$. Then,
\begin{align}
\|u(t)\|_{H^\gamma_x} \gtrsim (T^* -t)^{-\frac{\gamma}{4}}, \label{blowup criterion}
\end{align}
for all $0<t<T^*$.
\end{coro}
\begin{proof}
We follow the argument of \cite{CazenaveWeissler}. Let $0<t<T^*$. If we consider (NL4S) with initial data $u(t)$, then it follows from $(\ref{contraction estimate})$ the fixed point argument that if for some $M>0$ 
\[
C\|u(t)\|_{H^\gamma_x} + C(T-t)^{\frac{2\gamma}{d}} M^{1+\frac{8}{d}} \leq M,
\] 
then $T<T^*$. Thus, 
\[
C\|u(t)\|_{H^\gamma_x} + C(T^*-t)^{\frac{2\gamma}{d}} M^{1+\frac{8}{d}} > M,
\] 
for all $M>0$. Choosing $M=2C\|u(t)\|_{H^\gamma_x}$, we see that
\[
(T^* -t)^{\frac{2\gamma}{d}} \|u(t)\|^{\frac{8}{d}}_{H^\gamma_x} >C.
\]
This proves $(\ref{blowup criterion})$ and the proof is complete.
\end{proof}
We next define for any spacetime slab $J\times \R^d$, 
\begin{align}
Z_I(J):= \sup_{(p,q)\in B} \|\scal{\Delta} I u\|_{L^p_t(J, L^q_x)}. \nonumber
\end{align}
We have the following commutator estimates.
\begin{lem} \label{lem commutator estimate}
Let $5\leq d\leq 7, 1<\gamma<2, 0<\delta <\gamma-1$ and $J$ a compact interval. Then
\begin{align}
\|IF(u)\|_{L^2_t(J, L^{\frac{2d}{d+4}}_x)} &\lesssim |J|^{\frac{2\gamma}{d}} (Z_I(J))^{1+\frac{8}{d}}, \label{commutator estimate 1} \\
\|\nabla I F(u)-(I\nabla u) F'(u)\|_{L^2_t(J, L^{\frac{2d}{d+2}}_x)} &\lesssim N^{-(2-\gamma+\delta)} (Z_I(J))^{1+\frac{8}{d}}, \label{commutator estimate 2} \\
\|\nabla I F(u)\|_{L^2_t(J, L^{\frac{2d}{d+2}}_x)} &\lesssim  |J|^{\frac{2\gamma}{d}} (Z_I(J))^{1+\frac{8}{d}} + N^{-(2-\gamma+\delta)}(Z_I(J))^{1+\frac{8}{d}}, \label{commutator estimate 3} \\
\|\nabla I F(u)\|_{L^2_t(J, L^{\frac{2d}{d+4}}_x)} &\lesssim (Z_I(J))^{1+\frac{8}{d}}. \label{commutator estimate 4}
\end{align}
\end{lem}
\begin{proof} 
We firstly note that the estimates $(\ref{commutator estimate 2})$ and $(\ref{commutator estimate 4})$ are given in \cite[Lemma 3.1]{Dinhmass}. Let us consider $(\ref{commutator estimate 1})$. By $(\ref{property 1})$ and H\"older's inequality,
\[
\|IF(u)\|_{L^2_t(J, L^{\frac{2d}{d+4}}_x)} \lesssim \|F(u)\|_{L^2_t(J, L^{\frac{2d}{d+4}}_x)} \lesssim \|u\|_{L^{\frac{2(d+8)}{d-4\gamma}}_t(J, L^{\frac{2d(d+8)}{d^2+4d+16\gamma}}_x)} \|F'(u)\|_{L^{\frac{d+8}{2(2+\gamma)}}_t(J, L^{\frac{d(d+8)}{4d+16-8\gamma}}_x)}. 
\]
Since $F'(u)=O(|u|^{\frac{8}{d}})$, the Sobolev embedding implies
\begin{align*}
\|IF(u)\|_{L^2_t(J, L^{\frac{2d}{d+4}}_x)} &\lesssim \|u\|_{L^{\frac{2(d+8)}{d-4\gamma}}_t(J, L^{\frac{2d(d+8)}{d^2+4d+16\gamma}}_x)} \|u\|^{\frac{8}{d}}_{L^{\frac{4(d+8)}{d(2+\gamma)}}_t(J, L^{\frac{2(d+8)}{d+4-2\gamma}}_x)} \\
&\lesssim |J|^{\frac{2\gamma}{d}} \|u\|_{L^{\frac{2(d+8)}{d-4\gamma}}_t(J, L^{\frac{2d(d+8)}{d^2+4d+16\gamma}}_x)} \|u\|^{\frac{8}{d}}_{L^{\frac{2(d+8)}{d-4\gamma}}_t(J, L^{\frac{2(d+8)}{d+4-2\gamma}}_x)} \\
&\lesssim |J|^{\frac{2\gamma}{d}} \|u\|_{L^{\frac{2(d+8)}{d-4\gamma}}_t(J, L^{\frac{2d(d+8)}{d^2+4d+16\gamma}}_x)} \||\nabla|^\gamma u\|_{L^{\frac{2(d+8)}{d-4\gamma}}_t(J,L^{\frac{2d(d+8)}{d^2+4d+16\gamma}}_x)} \\
&\lesssim |J|^{\frac{2\gamma}{d}} \|\scal{\nabla}^\gamma u\|^{1+\frac{8}{d}}_{L^{\frac{2(d+8)}{d-4\gamma}}_t(J,L^{\frac{2d(d+8)}{d^2+4d+16\gamma}}_x)} \\
&\lesssim |J|^{\frac{2\gamma}{d}} (Z_I(J))^{1+\frac{8}{d}}.
\end{align*}
Here we use $(\ref{property 3})$ and the fact $\left( \frac{2(d+8)}{d-4\gamma}, \frac{2d(d+8)}{d^2+4d+16\gamma}\right)$ is biharmonic admissible to get the last estimate. \newline
\indent It remains to prove $(\ref{commutator estimate 3})$. We have from $(\ref{commutator estimate 2})$ and the triangle inequality that
\begin{align}
\|\nabla I F(u)\|_{L^2_t(J, L^{\frac{2d}{d+2}}_x)} \lesssim \|(\nabla I u) F'(u)\|_{L^2_t(J, L^{\frac{2d}{d+2}}_x)} + N^{-(2-\gamma+\delta)} (Z_I(J))^{1+\frac{8}{d}}. \label{estimate 1}
\end{align}
By H\"older's inequality,
\begin{align}
 \|(\nabla I u) F'(u)\|_{L^2_t(J, L^{\frac{2d}{d+2}}_x)} \lesssim \|\nabla I u\|_{L^{\frac{2(d+8)}{d-4\gamma}}_t(J, L^{\frac{2d(d+8)}{d^2+2d+16(\gamma-1)}}_x)} \|F'(u)\|_{L^{\frac{d+8}{2(2+\gamma)}}_t(J, L^{\frac{d(d+8)}{4d+16-8\gamma}}_x)}. \label{estimate 2}
\end{align}
We use the Sobolev embedding to estimate
\begin{align}
\|\nabla I u\|_{L^{\frac{2(d+8)}{d-4\gamma}}_t(J, L^{\frac{2d(d+8)}{d^2+2d+16(\gamma-1)}}_x)} \lesssim \|\Delta I u\|_{L^{\frac{2(d+8)}{d-4\gamma}}_t(J, L^{\frac{2d(d+8)}{d^2+4d+16\gamma}}_x)} \lesssim Z_I(J). \label{estimate 3}
\end{align}
Here $\Big(\frac{2(d+8)}{d-4\gamma}, \frac{2d(d+8)}{d^2+4d+16\gamma} \Big)$ is biharmonic admissible. Since $F'(u)=O(|u|^{\frac{8}{d}})$, the Sobolev embedding again gives
\begin{align}
\|F'(u)\|_{L^{\frac{d+8}{2(2+\gamma)}}_t(J, L^{\frac{d(d+8)}{4d+16-8\gamma}}_x)} &\lesssim \|u\|^{\frac{8}{d}}_{L^{\frac{4(d+8)}{d(2+\gamma)}}_t(J, L^{\frac{2(d+8)}{d+4-2\gamma}}_x)} \nonumber \\
&\lesssim |J|^{\frac{2\gamma}{d}} \|u\|^{\frac{8}{d}}_{L^{\frac{2(d+8)}{d-4\gamma}}_t(J, L^{\frac{2(d+8)}{d+4-2\gamma}}_x)} \nonumber \\
&\lesssim |J|^{\frac{2\gamma}{d}} \||\nabla|^\gamma u\|^{\frac{8}{d}}_{L^{\frac{2(d+8)}{d-4\gamma}}_t(J, L^{\frac{2d(d+8)}{d^2+4d+16\gamma}}_x)} \nonumber \\
&\lesssim |J|^{\frac{2\gamma}{d}} (Z_I(J))^{\frac{8}{d}}. \label{estimate 4}
\end{align}
Collecting $(\ref{estimate 1}) - (\ref{estimate 4})$, we obtain $(\ref{commutator estimate 3})$. The proof is complete.
\end{proof}
\begin{prop}[Modified local well-posedness] \label{prop modified local well-posedness}
Let $5\leq d \leq 7, 1<\gamma<2, 0<\delta<\gamma-1$ and $u_0 \in H^\gamma(\R^d)$. Let 
\[
\widetilde{T}_{\emph{lwp}}:= c \|I u_0\|_{H^2_x}^{-\frac{4}{\gamma}},
\]
for a small constant $c=c(d,\gamma)>0$. Then \emph{(NL4S)} is locally well-posed on $[0,\widetilde{T}_{\emph{lwp}}]$. Moreover, for $N$ sufficiently large,
\begin{align}
Z_I([0,\widetilde{T}_{\emph{lwp}}]) \lesssim \|I u_0\|_{H^2_x}. \label{modified local estimate}
\end{align}
\end{prop}
\begin{proof}
By $(\ref{property 4})$, $\|u_0\|_{H^\gamma_x}\lesssim \|Iu_0\|_{H^2_x}$. Thus, 
\[
\widetilde{T}_{\text{lwp}} = c\|I u_0\|_{H^2_x}^{-\frac{4}{\gamma}} \lesssim c \|u_0\|_{H^\gamma_x}^{-\frac{4}{\gamma}} \leq T_{\text{lwp}},
\]
provided $c$ is small enough. Here $T_{\text{lwp}}$ is as in Proposition $\ref{prop local well-posedness}$. This shows that (NL4S) is locally well-posed on $[0,\widetilde{T}_{\text{lwp}}]$. It remains to prove $(\ref{modified local estimate})$. Denote $J=[0,\widetilde{T}_{\text{lwp}}]$. By Strichartz estimates $(\ref{strichartz estimate biharmonic})$ and $(\ref{strichartz estimate biharmonic fourth-order})$, 
\begin{align*}
Z_I(J) &\lesssim \sup_{(p,q)\in B} \|Iu\|_{L^p_t(J, L^q_x)} + \sup_{(p,q)\in B} \|\Delta I u \|_{L^p_t(J, L^q_x)} \\
&\lesssim \|I u_0\|_{L^2_x} + \|IF(u)\|_{L^2_t(J, L^{\frac{2d}{d+4}}_x)} + \|\Delta I u_0\|_{L^2_x} + \|\nabla IF(u)\|_{L^2_t(J, L^{\frac{2d}{d+2}}_x)} \\
&\lesssim \|Iu_0\|_{H^2_x} + \|IF(u)\|_{L^2_t(J, L^{\frac{2d}{d+4}}_x)} + \|\nabla I F(u)\|_{L^2_t(J,L^{\frac{2d}{d+2}}_x)}.
\end{align*}
We next use $(\ref{commutator estimate 1})$ and $(\ref{commutator estimate 3})$ to have
\[
Z_I(J)\lesssim \|I u_0\|_{H^2_x} + \Big(|J|^{\frac{2\gamma}{d}} + N^{-(2-\gamma+\delta)}\Big) (Z_I(J))^{1+\frac{8}{d}}.
\]
By taking $c=c(d,\gamma)$ small enough (or $|J|$ is small) and $N$ large enough, the continuity argument shows $(\ref{modified local estimate})$. The proof is complete. 
\end{proof}
\section{Modified energy increment} \label{section modified energy increment}
\setcounter{equation}{0}
In this section, we will derive two types of the modified energy increment. The first one is to show that the modified energy of $u$, namely $E(Iu)$ grows much slower than the modified kinetic of $u$, namely $\|\Delta Iu\|^2_{L^2_x}$. It is crucial to prove the limiting profile for blowup solutions given in Theorem $\ref{theorem weak limiting profile}$. The second one is the ``almost'' conservation law for initial data whose mass is smaller than mass of the solution to the ground state equation $(\ref{ground state equation})$. With the help of this ``almost'' conservation law, we are able to prove the global well-posedness given in Theorem $\ref{theorem global existence below ground state}$.
\begin{lem}[Local increment of the modified energy] \label{lem local increment}
Let $5\leq d\leq 7, \max\{3-\frac{8}{d}, \frac{8}{d} \} <\gamma<2, 0<\delta <\gamma+\frac{8}{d}-3$ and $u_0 \in H^\gamma(\R^d)$. Let 
\[
\widetilde{T}_{\emph{lwp}}:=c\|Iu_0\|_{H^2_x}^{-\frac{4}{\gamma}},
\]
for some small constant $c=c(d,\gamma)>0$. Then, for $N$ sufficiently large,
\begin{align}
\sup_{t\in [0,\widetilde{T}_{\emph{lwp}}]} |E(I u(t))-E(Iu_0)| \lesssim N^{-(2-\gamma+\delta)} \Big( \|Iu_0\|_{H^2_x}^{2+\frac{8}{d}} + \|Iu_0\|_{H^2_x}^{2+\frac{16}{d}} \Big). \label{local energy increment}
\end{align}
Here the implicit constant depends only on $\gamma$ and $\|u_0\|_{H^\gamma_x}$.
\end{lem}
\begin{proof}
By Proposition $\ref{prop modified local well-posedness}$, the equation (NL4S) is locally well-posed on $[0,\widetilde{T}_{\text{lwp}}]$ and the unique solution $u$ satisfies
\begin{align}
Z_I([0,\widetilde{T}_{\text{lwp}}]) \lesssim \|Iu_0\|_{H^2_x}. \label{control size Z_I}
\end{align}
Next, we have from a direct computation that
\[
\partial_t E(Iu(t)) = \re{\int  \overline{I \partial_t u} (\Delta^2 I u - F(I u))  dx }.
\]
The Fundamental Theorem of Calculus gives
\[
E(Iu(t))- E(Iu_0) = \int_0^t \partial_s E(Iu(s))ds = \re{\int_0^t \int \overline{I \partial_s u } (\Delta^2 I u - F(Iu)) dxds  }.
\]
As $I \partial_t u = i \Delta^2 I u -i I F(u)$, we have
\begin{align*}
E(Iu(t))-E(Iu_0) &= \re{\int_0^t \int \overline{I \partial_s}(IF(u)-F(Iu) dxds }  \\
&= \im{\int_0^t \int \overline{\Delta^2 I u - IF(u)}(IF(u)-F(Iu)   dx ds} \\
&= \im{\int_0^t \int \overline{\Delta I u} \Delta(IF(u)-F(Iu)) dxds} \\
&\mathrel{\phantom{=}} - \im{\int_0^t \int \overline{IF(u)}(IF(u)-F(Iu))  dx ds}.
\end{align*}
We next write
\begin{align*}
\Delta (IF(u)-F(Iu)) &= I(\Delta u F'(u) + |\nabla u|^2 F''(u) ) -\Delta Iu F'(Iu) - |\nabla I u|^2 F''(Iu) \\
&=\Delta I u (F'(u)-F'(Iu)) + |\nabla I u|^2(F''(u)-F''(Iu)) + \nabla I u \cdot(\nabla u - \nabla I u) F''(u) \\
& \mathrel{\phantom{=}} + I(\Delta F'(u))- (\Delta I u) F'(u) + I(\nabla u \cdot \nabla u F''(u)) - (\nabla I u) \cdot \nabla u F''(u).
\end{align*}
Thus,
\begin{align}
E(Iu(t))-E(Iu_0) &= \im{\int_0^t \int \overline{\Delta I u} \Delta I u (F'(u)-F'(Iu))   dxds} \label{energy increment 1} \\
& \mathrel{\phantom{=}} + \im{\int_0^t \int \overline{\Delta I u} |\nabla I u|^2 (F''(u)-F''(Iu))   dxds} \label{energy increment 2} \\
&\mathrel{\phantom{=}} + \im{\int_0^t \int \overline{\Delta I u} \nabla I u \cdot (\nabla u -\nabla I u) F''(u)   dxds} \label{energy increment 3} \\
&\mathrel{\phantom{=}} + \im{\int_0^t \int \overline{\Delta I u} [I(\Delta u F'(u)) - (\Delta I u) F'(u)] dxds} \label{energy increment 4} \\
&\mathrel{\phantom{=}} + \im{\int_0^t \int \overline{\Delta I u} [I(\nabla u \cdot \nabla u F''(u)) - (\nabla I u) \cdot \nabla u F''(u)] dxds} \label{energy increment 5} \\
&\mathrel{\phantom{=}} -\im{\int_0^t \int \overline{IF(u)}(IF(u)-F(Iu))  dx ds}. \label{energy increment 6}
\end{align}
Let $J=[0,\widetilde{T}_{\text{lwp}}]$. By H\"older's inequality, we estimate
\begin{align}
|(\ref{energy increment 1})| &\lesssim \|\Delta I u\|^2_{L^4_t(J, L^{\frac{2d}{d-2}}_x)} \|F'(u)-F'(Iu)\|_{L^2_t(J, L^{\frac{d}{2}}_x)} \nonumber \\
& \lesssim (Z_I(J))^2 \| |u-Iu|(|u|+|Iu|)^{\frac{8}{d}-1} \|_{L^2_t(J, L^{\frac{d}{2}}_x)} \nonumber \\
& \lesssim (Z_I(J))^2 \|P_{>N} u\|_{L^{\frac{16}{d}}_t(J, L^4_x)} \|u\|^{\frac{8}{d}-1}_{L^{\frac{16}{d}}_t(J, L^4_x)}. \label{energy increment 1 sub 1}
\end{align}
By $(\ref{property 2})$, 
\begin{align}
\|P_{>N} u\|_{L^{\frac{16}{d}}_t(J, L^4_x)} \lesssim N^{-2} \|\Delta I u\|_{L^{\frac{16}{d}}_t(J, L^4_x)} \lesssim N^{-2} Z_I(J). \label{energy increment 1 sub 2}
\end{align}
Here $\Big(\frac{16}{d}, 4 \Big)$ is biharmonic admissible. Similarly, by $(\ref{property 3})$,
\begin{align}
\|u\|_{L^{\frac{16}{d}}_t(J, L^4_x)} \lesssim Z_I(J). \label{energy increment 1 sub 3}
\end{align}
Collecting $(\ref{energy increment 1 sub 1})-(\ref{energy increment 1 sub 3})$, we get
\begin{align}
|(\ref{energy increment 1})| \lesssim N^{-2} (Z_I(J))^{2+\frac{8}{d}}. \label{energy increment 1 final}
\end{align}
Next, we bound
\begin{align}
|(\ref{energy increment 2})| &\lesssim \|\Delta I u\|_{L^4_t(J, L^{\frac{2d}{d-2}}_x)} \||\nabla I u|^2\|_{L^{\frac{16}{11}}_t(J, L^{\frac{4d}{4d-11}}_x)} \|F''(u)-F''(Iu)\|_{L^{16}_t(J, L^{\frac{4d}{15-2d}}_x)} \nonumber \\
& \lesssim \|\Delta I u\|_{L^4_t(J, L^{\frac{2d}{d-2}}_x)} \|\nabla I u\|^2_{L^{\frac{32}{11}}_t(J, L^{\frac{8d}{4d-11}}_x)} \|F''(u)-F''(Iu)\|_{L^{16}_t(J, L^{\frac{4d}{15-2d}}_x)} \nonumber \\
& \lesssim (Z_I(J))^3 \||u-Iu|^{\frac{8}{d}-1} \|_{L^{16}_t(J, L^{\frac{4d}{15-2d}}_x) } \nonumber \\
&\lesssim (Z_I(J))^3 \|P_{>N} u\|^{\frac{8}{d}-1}_{L^{\frac{16(8-d)}{d}}_t(J, L^{\frac{4(8-d)}{15-2d}}_x)} \nonumber\\
&\lesssim N^{-2\left(\frac{8}{d}-1\right)} (Z_I(J))^{2+\frac{8}{d}}. \label{energy increment 2 final}
\end{align}
The third line follows by dropping the $I$-operator and applying $(\ref{property 3})$ with the fact $\gamma>1$. We also use the fact 
\[
|F''(z)-F''(\zeta)| \lesssim |z-\zeta|^{\frac{8}{d}-1}, \quad \forall z, \zeta \in \C,
\] 
for $5\leq d \leq 7$. The last estimate uses $(\ref{energy increment 1 sub 2})$. Note that $\Big(\frac{32}{11}, \frac{8d}{4d-11}\Big)$ and $\Big(\frac{16(8-d)}{d}, \frac{4(8-d)}{15-2d}\Big)$ are biharmonic admissible. Similarly, we estimate
\begin{align*}
|(\ref{energy increment 3})| &\lesssim \|\Delta I u\|_{L^4_t(J, L^{\frac{2d}{d-2}}_x)} \|\nabla I u\|_{L^{\frac{32}{11}}_t(J, L^{\frac{8d}{4d-11}}_x)} \|\nabla u- \nabla Iu\|_{L^{\frac{32}{11}}_t(J, L^{\frac{8d}{4d-11}}_x)}  \|F''(u)\|_{L^{16}_t(J, L^{\frac{4d}{15-2d}}_x)} \nonumber \\
&\lesssim (Z_I(J))^2 \|\nabla P_{>N} u\|_{L^{\frac{32}{11}}_t(J, L^{\frac{8d}{4d-11}}_x)} \|F''(u)\|_{L^{16}_t(J, L^{\frac{4d}{15-2d}}_x)}.
\end{align*}
Using $(\ref{property 2})$, we have
\[
\|\nabla P_{>N} u\|_{L^{\frac{32}{11}}_t(J, L^{\frac{8d}{4d-11}}_x)} \lesssim N^{-1} \|\Delta I u\|_{L^{\frac{32}{11}}_t(J, L^{\frac{8d}{4d-11}}_x)} \lesssim N^{-1} Z_I(J).
\]
As $F''(u)=O(|u|^{\frac{8}{d}-1})$, the estimate $(\ref{property 3})$ gives
\begin{align}
\|F''(u)\|_{L^{16}_t(J, L^{\frac{4d}{15-2d}}_x)} \lesssim \|u\|^{\frac{8}{d}-1}_{L^{\frac{16(8-d)}{d}}_t(J, L^{\frac{4(8-d)}{15-2d}}_x)} \lesssim (Z_I(J))^{\frac{8}{d}-1}. \label{estimate second derivative}
\end{align}
We thus obtain 
\begin{align}
|(\ref{energy increment 3})| \lesssim N^{-1} (Z_I(J))^{2+\frac{8}{d}}. \label{energy increment 3 final}
\end{align}
By H\"older's inequality,
\begin{align}
|(\ref{energy increment 4})| \lesssim \|\Delta I u\|_{L^2_t(J, L^{\frac{2d}{d-4}}_x)} \|I(\Delta u F'(u))-(\Delta I u) F'(u)\|_{L^2_t(J, L^{\frac{2d}{d+4}}_x)}. \label{energy increment 4 sub 1}
\end{align}
We then apply Lemma $\ref{lem rougher estimate}$ with $q=\frac{2d}{d+4}, q_1 = \frac{2d(d-3)}{d^2-7d+16}$ and $q_2=\frac{d(d-3)}{2(2d-7)}$ to get
\[
\|I(\Delta u F'(u))-(\Delta I u) F'(u)\|_{L^{\frac{2d}{d+4}}_x} \lesssim N^{-\alpha} \|\Delta I u\|_{L^{\frac{2d(d-3)}{d^2-7d+16}}_x} \|\scal{\nabla}^\alpha F'(u)\|_{L^{\frac{d(d-3)}{2(2d-7)}}_x},
\]
where $\alpha=2-\gamma+\delta$. The H\"older inequality then implies
\begin{multline}
\|I(\Delta u F'(u))-(\Delta I u) F'(u)\|_{L^2_t(J, L^{\frac{2d}{d+4}}_x)} \lesssim N^{-\alpha} \|\Delta I u\|_{L^{\frac{2(d-3)}{d-4}}_t(J, L^{\frac{2d(d-3)}{d^2-7d+16}}_x)} \\
\times \|\scal{\nabla}^\alpha F'(u)\|_{L^{2(d-3)}_t(J, L^{\frac{d(d-3)}{2(2d-7)}}_x)}. \label{energy increment 4 sub 2}
\end{multline}
We have
\begin{align}
\|\scal{\nabla}^\alpha F'(u) \|_{L^{2(d-3)}_t(J, L^{\frac{d(d-3)}{2(2d-7)}}_x)} \lesssim \|F'(u)\|_{L^{2(d-3)}_t(J, L^{\frac{d(d-3)}{2(2d-7)}}_x)} + \||\nabla|^\gamma F'(u)\|_{L^{2(d-3)}_t(J, L^{\frac{d(d-3)}{2(2d-7)}}_x)}. \label{first derivative}
\end{align}
As $F'(u)=O(|u|^{\frac{8}{d}})$, the estimate $(\ref{property 3})$ implies
\begin{align}
\|F'(u)\|_{L^{2(d-3)}_t(J, L^{\frac{d(d-3)}{2(2d-7)}}_x)} \lesssim \|u\|^{\frac{8}{d}}_{L^{\frac{16(d-3)}{d}}_t(J, L^{\frac{4(d-3)}{2d-7}}_x)} \lesssim (Z_I(J))^{\frac{8}{d}}. \label{first derivative term 1}
\end{align}
Here $\Big(\frac{16(d-3)}{d}, \frac{4(d-3)}{2d-7}\Big)$ is biharmonic admissible. In order to treat the second term in $(\ref{first derivative})$, we apply Lemma $\ref{lem fractional chain}$ with $q=\frac{d(d-3)}{2(2d-7)}, q_1 = \frac{2d(d-3)}{-d^2+11d-26}$ and $q_2=\frac{2d(d-3)}{d^2-3d-2}$ to get
\begin{align}
\||\nabla|^\alpha F'(u)\|_{L^{\frac{d(d-3)}{2(2d-7)}}_x} \lesssim \|F''(u)\|_{L^{\frac{2d(d-3)}{-d^2+11d-26}}_x} \||\nabla|^\alpha u\|_{L^{\frac{2d(d-3)}{d^2-3d-2}}_x}. \label{first derivative term 2}
\end{align}
H\"older's inequality then gives
\[
\||\nabla|^\alpha F'(u)\|_{L^{2(d-3)}_t(J,L^{\frac{d(d-3)}{2(2d-7)}}_x)} \lesssim \|F''(u)\|_{L^{4(d-3)}_t(J, L^{\frac{2d(d-3)}{-d^2+11d-26}}_x)} \||\nabla|^\alpha u\|_{L^{4(d-3)}_t(J, L^{\frac{2d(d-3)}{d^2-3d-2}}_x)}.
\]
As $F''(u)=O(|u|^{\frac{8}{d}-1})$, we have
\begin{align}
\|F''(u)\|_{L^{4(d-3)}_t(J, L^{\frac{2d(d-3)}{-d^2+11d-26}}_x)} \lesssim \|u\|^{\frac{8}{d}-1}_{L^{\frac{4(8-d)(d-3)}{d}}_t(J, L^{\frac{2(8-d)(d-3)}{-d^2+11d-26}}_x)} \lesssim (Z_I(J))^{\frac{8}{d}-1}. \label{first derivative term 3}
\end{align}
Here $\Big( \frac{4(8-d)(d-3)}{d}, \frac{2(8-d)(d-3)}{-d^2+11d-26} \Big)$ is biharmonic admissible. Since $\Big(4(d-3), \frac{2d(d-3)}{d^2-3d-2} \Big)$ is also a biharmonic admissible, we have from $(\ref{property 3})$ that
\begin{align}
\||\nabla|^\alpha u\|_{L^{4(d-3)}_t(J, L^{\frac{2d(d-3)}{d^2-3d-2}}_x)} \lesssim Z_I(J). \label{first derivative term 4}
\end{align}
Note that $\alpha<1<\gamma$. Collecting $(\ref{first derivative}) - (\ref{first derivative term 4})$, we show
\begin{align}
\|\scal{\nabla}^\alpha F'(u) \|_{L^{2(d-3)}_t(J, L^{\frac{d(d-3)}{2(2d-7)}}_x)} \lesssim (Z_I(J))^{\frac{8}{d}}. \label{energy increment 4 sub 3}
\end{align}
Combining $(\ref{energy increment 4 sub 1}), (\ref{energy increment 4 sub 2})$ and $(\ref{energy increment 4 sub 3})$, we get
\begin{align}
|(\ref{energy increment 4})| \lesssim N^{-(2-\gamma+\delta)} (Z_I(J))^{2+\frac{8}{d}}. \label{energy increment 4 final}
\end{align}
Similarly, we bound
\begin{align}
|(\ref{energy increment 5})| &\lesssim \|\Delta I u\|_{L^4_t(J, L^{\frac{2d}{d-2}}_x)} \|I(\nabla u \cdot \nabla u F''(u))-(I\nabla u)\cdot \nabla u F''(u)\|_{L^{\frac{4}{3}}_t(J, L^{\frac{2d}{d+2}}_x)}. \label{energy increment 5 sub 1}
\end{align}
Applying Lemma $\ref{lem rougher estimate}$ with $q=\frac{2d}{d+2}, q_1 = \frac{8d}{4d-11}$ and $q_2=\frac{8d}{19}$ and using H\"older inequality, we have
\begin{multline}
\|I(\nabla u \cdot \nabla u F''(u))-(I\nabla u)\cdot \nabla u F''(u)\|_{L^{\frac{4}{3}}_t(J, L^{\frac{2d}{d+2}}_x)} \lesssim N^{-\alpha} \|I\nabla u\|_{L^{\frac{32}{11}}_t(J, L^{\frac{8d}{4d-11}}_x)} \\
\times \|\scal{\nabla}^\alpha (\nabla u F''(u)) \|_{L^{\frac{8}{5}}_t(J, L^{\frac{8d}{19}}_x)}. \label{energy increment 5 sub 2}
\end{multline}
The fractional chain rule implies
\begin{multline}
\|\scal{\nabla}^\alpha (\nabla u F''(u)) \|_{L^{\frac{8}{5}}_t(J, L^{\frac{8d}{19}}_x)} \lesssim \|\scal{\nabla}^{\alpha+1} u\|_{L^{\frac{32}{11}}_t(J, L^{\frac{8d}{4d-11}}_x)} \|F''(u)\|_{L^{16}_t(J, L^{\frac{4d}{15-2d}}_x)} \\
+ \|\nabla u\|_{L^{\frac{32}{11}}_t(J, L^{\frac{8d}{4d-11}}_x)} \|\scal{\nabla}^\alpha F''(u)\|_{L^{16}_t(J, L^{\frac{4d}{15-2d}}_x)}. \label{energy increment 5 sub 3}
\end{multline}
By our assumptions on $\gamma$ and $\delta$, we see that $\alpha+1<\gamma$. Thus, using $(\ref{property 3})$ (and dropping the $I$-operator if necessary) and $(\ref{estimate second derivative})$, we have
\begin{align}
\|I\nabla u\|_{L^{\frac{32}{11}}_t(J, L^{\frac{8d}{4d-11}}_x)}, \|\nabla u\|_{L^{\frac{32}{11}}_t(J, L^{\frac{8d}{4d-11}}_x)}, \|\scal{\nabla}^{\alpha+1} u\|_{L^{\frac{32}{11}}_t(J, L^{\frac{8d}{4d-11}}_x)} &\lesssim Z_I(J), \label{energy increment 5 sub 4} \\
\|F''(u)\|_{L^{16}_t(J, L^{\frac{4d}{15-2d}}_x)} &\lesssim (Z_I(J))^{\frac{8}{d}-1}. \label{energy increment 5 sub 5}
\end{align}
Here $\left(\frac{32}{11}, \frac{8d}{4d-11}\right)$ is biharmonic admissible. It remains to bound $\|\scal{\nabla}^\alpha F''(u)\|_{L^{16}_t(J, L^{\frac{4d}{15-2d}}_x)}$. To do so, we use
\begin{align}
\|\scal{\nabla}^\alpha F''(u)\|_{L^{16}_t(J, L^{\frac{4d}{15-2d}}_x)} \lesssim \|F''(u)\|_{L^{16}_t(J, L^{\frac{4d}{15-2d}}_x)} + \||\nabla|^\alpha F''(u)\|_{L^{16}_t(J, L^{\frac{4d}{15-2d}}_x)}. \label{energy increment 5 sub 6}
\end{align}
We next use Lemma $\ref{lem fractional chain rule holder}$ with $\beta=\frac{8}{d}-1$, $\alpha=2-\gamma+\delta$, $q=\frac{4d}{15-2d}$ and $q_1, q_2$ satisfying
\[
\Big(\frac{8}{d}-1 -\frac{\alpha}{\rho}\Big) q_1 = \frac{\alpha}{\rho}q_2=\frac{4(8-d)}{15-2d},
\]
and $\frac{\alpha}{\frac{8}{d}-1}<\rho<1$. Note that the choice of $\rho$ is possible since $\alpha <\frac{8}{d}-1$ by our assumptions. With these choices, we have
\[
\Big(1-\frac{\alpha}{\beta \rho}\Big) q_1 = \frac{4d}{15-2d}>1,
\]
for $5\leq d\leq 7$. Then,
\begin{align*}
\||\nabla|^\alpha F''(u)\|_{L^{\frac{4d}{15-2d}}_x} \lesssim \||u|^{\frac{8}{d}-1-\frac{\alpha}{\rho}} \|_{L^{q_1}_x} \||\nabla|^\rho u\|^{\frac{\alpha}{\rho}}_{L^{\frac{\alpha}{\rho} q_2}_x} \lesssim \|u\|^{\frac{8}{d}-1-\frac{\alpha}{\rho}}_{L^{\left(\frac{8}{d}-1-\frac{\alpha}{\rho}\right)q_1}_x} \||\nabla|^\rho u\|^{\frac{\alpha}{\rho}}_{L^{\frac{\alpha}{\rho} q_2}_x}.
\end{align*}
By H\"older's inequality,
\begin{align*}
\||\nabla|^\alpha F''(u)\|_{L^{16}_t(J,L^{\frac{4d}{15-2d}}_x)} &\lesssim  \|u\|^{\frac{8}{d}-1-\frac{\alpha}{\rho}}_{L^{\left(\frac{8}{d}-1-\frac{\alpha}{\rho}\right)p_1}_t(J, L^{\left(\frac{8}{d}-1-\frac{\alpha}{\rho}\right)q_1}_x)} \||\nabla|^\rho u\|^{\frac{\alpha}{\rho}}_{L^{\frac{\alpha}{\rho} p_2}_t(J,L^{\frac{\alpha}{\rho} q_2}_x)} \\
&=\|u\|^{\frac{8}{d}-1-\frac{\alpha}{\rho}}_{L^{\frac{16(8-d)}{d}}_t(J,L^{\frac{4(8-d)}{15-2d}}_x)} \||\nabla|^\rho u\|^{\frac{\alpha}{\rho}}_{L^{\frac{16(8-d)}{d}}_t(J,L^{\frac{4(8-d)}{15-2d}}_x)},
\end{align*}
provided
\[
\Big(\frac{8}{d}-1 -\frac{\alpha}{\rho}\Big) p_1 = \frac{\alpha}{\rho}p_2=\frac{16(8-d)}{d}.
\]
Since $\left(\frac{16(8-d)}{d}, \frac{4(8-d)}{15-2d}\right)$ is biharmonic admissible, we have from $(\ref{property 3})$ with the fact $0<\rho<1<\gamma$ that
\begin{align}
\||\nabla|^\alpha F''(u)\|_{L^{16}_t(J,L^{\frac{4d}{15-2d}}_x)} \lesssim (Z_I(J))^{\frac{8}{d}-1}. \label{energy increment 5 sub 7}
\end{align}
Collecting $(\ref{energy increment 5 sub 1}) - (\ref{energy increment 5 sub 7})$, we get
\begin{align}
|(\ref{energy increment 5})| \lesssim N^{-(2-\gamma+\delta)} (Z_I(J))^{2+\frac{8}{d}}. \label{energy increment 5 final}
\end{align}
Finally, we consider $(\ref{energy increment 6})$. We bound
\begin{align}
|(\ref{energy increment 6})| &\lesssim \||\nabla|^{-1} IF(u)\|_{L^2_t(J, L^{\frac{2d}{d-2}}_x)} \|\nabla(IF(u)-F(Iu))\|_{L^2_t(J, L^{\frac{2d}{d+2}}_x)} \nonumber \\
&\lesssim \|\nabla I F(u)\|_{L^2_t(J, L^{\frac{2d}{d+2}}_x)} \|\nabla(IF(u)-F(Iu))\|_{L^2_t(J, L^{\frac{2d}{d+2}}_x)}.  \label{energy increment 6 sub 1}
\end{align}
By $(\ref{commutator estimate 4})$, 
\[
\|\nabla I F(u)\|_{L^2_t(J, L^{\frac{2d}{d+2}}_x)} \lesssim (Z_I(J))^{1+\frac{8}{d}}.
\]
By the triangle inequality, we estimate
\begin{multline*}
\|\nabla(IF(u)-F(Iu))\|_{L^2_t(J, L^{\frac{2d}{d+2}}_x)} \lesssim \|(\nabla I u) (F'(u)-F'(Iu))\|_{L^2_t(J, L^{\frac{2d}{d+2}}_x)} \\
 + \|\nabla I F(u)-(\nabla I u) F'(u)\|_{L^2_t(J, L^{\frac{2d}{d+2}}_x)}.
\end{multline*}
We firstly use H\"older's inequality and  estimate as in $(\ref{energy increment 1 sub 1})$ to get
\begin{align}
\|(\nabla I u)(F'(u)-F'(Iu))\|_{L^2_t(J, L^{\frac{2d}{d+2}}_x)} &\lesssim \|\nabla I u\|_{L^\infty_t(J, L^{\frac{2d}{d-2}}_x)} \|F'(u)-F'(Iu)\|_{L^2_t(J, L^{\frac{d}{2}}_x)} \nonumber \\
&\lesssim \|\Delta I u\|_{L^\infty_t(J, L^2_x)} \|P_{>N} u\|_{L^{\frac{16}{d}}_t(J, L^4_x)} \|u\|^{\frac{8}{d}-1}_{L^{\frac{16}{d}}_t(J, L^4_x)} \nonumber \\
&\lesssim N^{-2} (Z_I(J))^{1+\frac{8}{d}}. \label{energy increment 6 sub 2}
\end{align}
By $(\ref{commutator estimate 2})$, 
\begin{align}
\|\nabla I F(u)-(\nabla I u) F'(u)\|_{L^2_t(J, L^{\frac{2d}{d+2}}_x)} \lesssim N^{-(2-\gamma+\delta)} (Z_I(J))^{1+\frac{8}{d}}. \label{energy increment 6 sub 3}
\end{align}
Combining $(\ref{energy increment 6 sub 1})-(\ref{energy increment 6 sub 3})$, we get
\begin{align}
|(\ref{energy increment 6})| &\lesssim (Z_I(J))^{1+\frac{8}{d}}(N^{-2} (Z_I(J))^{1+\frac{8}{d}} + N^{-(2-\gamma+\delta)} (Z_I(J))^{1+\frac{8}{d}}) \nonumber \\
&\lesssim N^{-(2-\gamma+\delta)}(Z_I(J))^{2+\frac{16}{d}}. \label{energy increment 6 final}
\end{align}
Combining $(\ref{energy increment 1 final}), (\ref{energy increment 2 final}), (\ref{energy increment 3 final}), (\ref{energy increment 4 final}), (\ref{energy increment 5 final}),  (\ref{energy increment 6 final})$ and using $(\ref{control size Z_I})$, we prove $(\ref{local energy increment})$. The proof is complete.
\end{proof}
We next introduce some notations. We define
\begin{align}
\Lambda(t):=\sup_{0\leq s\leq t}\|u(s)\|_{H^\gamma_x}, \quad \Sigma(t):=\sup_{0\leq s\leq t} \|I_{N}u(s)\|_{H^2_x}. \label{energy increment notations}
\end{align}
\begin{prop}[Increment of the modified energy] \label{prop increment modified energy}
Let $5\leq d \leq 7$ and $\frac{56-3d+\sqrt{137d^2+1712d+3136}}{2(2d+32)}<\gamma<2$.
Let $u_0\in H^\gamma(\R^d)$ be such that the corresponding solution $u$ to \emph{(NL4S)} blows up at time $0<T^* <\infty$. Let $0<T<T^*$. Then for 
\begin{align}
N(T)\sim \Lambda(T)^{\frac{a(\gamma)}{2(2-\gamma)}}, \label{define N_T}
\end{align}
we have 
\[
|E(I_{N(T)} u(T))| \lesssim \Lambda(T)^{a(\gamma)}.
\]
Here the implicit constants depend only on $\gamma, T^*$ and $\|u_0\|_{H^\gamma_x}$, and $0<a(\gamma)<2$ is given by
\begin{align}
a(\gamma):=\frac{2\left(2+\frac{16}{d}+\frac{4}{\gamma}\right)(2-\gamma)}{\left[\frac{8}{d}-1-(2-\gamma)\left(\frac{16}{d}+\frac{4}{\gamma}\right)\right]-}. \label{define a gamma}
\end{align}
\end{prop}
\begin{proof}
Let $\tau:=c\Sigma(T)^{-\frac{4}{\gamma}}$ for some constant $c=c(d,\gamma)>0$ small enough. For $N(T)$ sufficiently large, Proposition $\ref{prop modified local well-posedness}$ shows the local existence and the unique solution satisfies
\[
Z_{I_{N(T)}}([t,t+\tau]) \lesssim \|I_{N(T)} u(t)\|_{H^2_x} \lesssim \Sigma(T),
\]
uniformly in $t$ provided that $[t,t+\tau] \subset [0,T]$. We next split $[0,T]$ into $O(T/\tau)$ subintervals and apply Lemma $\ref{lem local increment}$ on each of these intervals to have
\begin{align}
\sup_{t\in[0,T]} |E(I_{N(T)}u(t))| &\lesssim |E(I_{N(T)} u_0)| + \frac{T}{\tau} N(T)^{-(2-\gamma+\delta)} \Big(\Sigma(T)^{2+\frac{8}{d}} + \Sigma(T)^{2+\frac{16}{d}}\Big) \\
&\lesssim |E(I_{N(T)} u_0)| + N(T)^{-(2-\gamma+\delta)} \Big(\Sigma(T)^{2+\frac{8}{d}+\frac{4}{\gamma}} + \Sigma(T)^{2+\frac{16}{d}+\frac{4}{\gamma}}\Big), \label{increment estimate 1}
\end{align}
for $\max\left\{3-\frac{8}{d},\frac{8}{d}\right\}<\gamma<2$ and $0<\delta<\gamma+\frac{8}{d}-3$. Next, by $(\ref{property 4})$, we have
\begin{align}
\Sigma(T) \lesssim N(T)^{2-\gamma} \Lambda(T). \label{increment estimate 2}
\end{align}
Moreover, the Gagliardo-Nirenberg inequality $(\ref{sharp gargliardo nirenberg inequality})$ together with $(\ref{property 5})$ imply
\begin{align}
|E(I_{N(T)} u_0)| &\lesssim \|\Delta I_{N(T)} u_0\|^2_{L^2_x} + \|I_{N(T)} u_0\|^{2+\frac{8}{d}}_{L^{2+\frac{8}{d}}_x} \nonumber \\
& \lesssim \|\Delta I_{N(T)} u_0\|^2_{L^2_x} + \|I_{N(T)} u_0\|^{\frac{8}{d}}_{L^2_x} \|\Delta I_{N(T)} u_0\|^2_{L^2_x} \nonumber \\
&\lesssim N(T)^{2(2-\gamma)} \Big(\|u_0\|_{H^\gamma_x}^2 +\|u_0\|_{H^\gamma_x}^{2+\frac{8}{d}}\Big) \nonumber \\
&\lesssim N^{2(2-\gamma)}. \label{increment estimate 3}
\end{align}
Substituting $(\ref{increment estimate 2})$ and $(\ref{increment estimate 3})$ to $(\ref{increment estimate 1})$, we get
\begin{multline}
\sup_{t\in[0,T]} |E(I_{N(T)}u(t))| \lesssim N(T)^{2(2-\gamma)} + N(T)^{-(2-\gamma+\delta) +(2-\gamma)\left(2+\frac{8}{d}+\frac{4}{\gamma}\right) } \Lambda(T)^{2+\frac{8}{d}+\frac{4}{\gamma}} \\
+ N(T)^{-(2-\gamma+\delta) +(2-\gamma)\left(2+\frac{16}{d}+\frac{4}{\gamma}\right) } \Lambda(T)^{2+\frac{16}{d}+\frac{4}{\gamma}}. \label{increment estimate 4}
\end{multline}
Optimizing $(\ref{increment estimate 4})$, we observe that if we take
\[
N(T)^{2(2-\gamma)} \sim N(T)^{-(2-\gamma+\delta) +(2-\gamma)\left(2+\frac{16}{d}+\frac{4}{\gamma}\right)} \Lambda(T)^{2+\frac{16}{d}+\frac{4}{\gamma}},
\]
or
\[
N(T)\sim \Lambda(T)^{\frac{2+\frac{16}{d}+\frac{4}{\gamma}}{(2-\gamma+\delta)-(2-\gamma)\left(\frac{16}{d}+\frac{4}{\gamma}\right)}},
\]
then
\[
\sup_{t\in[0,T]}|E(I_{N(T)} u(t))| \lesssim N(T)^{2(2-\gamma)} \sim \Lambda(T)^{\frac{2\left(2+\frac{16}{d}+\frac{4}{\gamma}\right)(2-\gamma)}{(2-\gamma+\delta)-(2-\gamma)\left(\frac{16}{d}+\frac{4}{\gamma}\right)}}.
\]
Denote
\[
a(\gamma):=\frac{2\left(2+\frac{16}{d}+\frac{4}{\gamma}\right)(2-\gamma)}{(2-\gamma+\delta)-(2-\gamma)\left(\frac{16}{d}+\frac{4}{\gamma}\right)}.
\]
Since $2-\gamma+\delta <\frac{8}{d}-1$, we see that
\[
a(\gamma)=\frac{2\left(2+\frac{16}{d}+\frac{4}{\gamma}\right)(2-\gamma)}{\left[\frac{8}{d}-1-(2-\gamma)\left(\frac{16}{d}+\frac{4}{\gamma}\right)\right]-}.
\]
In order to make $0<a(\gamma)<2$, we need
\begin{align}
\left\{
\begin{array}{rcl}
\frac{8}{d}-1- (2-\gamma)\Big(\frac{16}{d}+\frac{4}{\gamma}\Big)&> &0, \\
\Big(2+\frac{16}{d}+\frac{4}{\gamma}\Big)(2-\gamma)&<&\frac{8}{d}-1- (2-\gamma)\Big(\frac{16}{d}+\frac{4}{\gamma}\Big). 
\end{array}
\right. \label{condition gamma}
\end{align}
Solving $(\ref{condition gamma})$, we obtain
\[
\gamma>\frac{56-3d+\sqrt{137d^2+1712d+3136}}{2(2d+32)}.
\]
This completes the proof.
\end{proof}
\begin{prop}[Almost conservation law] \label{prop almost conservation law}
Let $5\leq d\leq 7, \max\{3-\frac{8}{d}, \frac{8}{d}\}<\gamma<2$ and $0<\delta<\gamma+\frac{8}{d}-3$. Let $u_0\in H^\gamma(\R^d)$ satisfying $\|u_0\|_{L^2_x} <\|Q\|_{L^2_x}$, where $Q$ is the solution to the ground state equation $(\ref{ground state equation})$. Assume in addition that $E(I u_0)\leq 1$. Let 
\[
\widetilde{T}_{\emph{lwp}}:=c \|Iu_0\|^{-\frac{4}{\gamma}}_{H^2_x},
\] 
for some small constant $c=c(d,\gamma)>0$. Then, for $N$ sufficiently large,
\[
\sup_{t\in[0,\widetilde{T}_{\emph{lwp}}]} |E(Iu(t)) - E(Iu_0)| \lesssim N^{-(2-\gamma+\delta)}.
\]
Here the implicit constant depends only on $\gamma$ and $E(Iu_0)$.
\end{prop}
\begin{rem} \label{rem energy below ground state}
Using the sharp Gagliardo-Nirenberg inequality together with the conservation of mass, the modified energy is always positive for initial data satisfying $\|u_0\|_{L^2_x}<\|Q\|_{L^2_x}$. Indeed,
\begin{align*}
E(Iu(t)) &=\frac{1}{2}\|\Delta I u(t)\|^2_{L^2_x} -\frac{1}{2+\frac{8}{d}}\|Iu(t)\|^{2+\frac{8}{d}}_{L^{2+\frac{8}{d}}_x} \\
&\geq \frac{1}{2}\|\Delta I u(t)\|^2_{L^2_x}-\frac{1}{2} \Big(\frac{\|Iu(t)\|_{L^2_x}}{\|Q\|_{L^2_x}} \Big)^{\frac{8}{d}} \|\Delta I u(t)\|^2_{L^2_x} \\
&\geq \frac{1}{2}\|\Delta I u(t)\|^2_{L^2_x}-\frac{1}{2} \Big(\frac{\|u(t)\|_{L^2_x}}{\|Q\|_{L^2_x}} \Big)^{\frac{8}{d}} \|\Delta I u(t)\|^2_{L^2_x} \\
&\geq \frac{1}{2}\|\Delta I u(t)\|^2_{L^2_x}-\frac{1}{2} \Big(\frac{\|u_0\|_{L^2_x}}{\|Q\|_{L^2_x}} \Big)^{\frac{8}{d}} \|\Delta I u(t)\|^2_{L^2_x}\\
&>0.
\end{align*} 
Here we use the fact that $\|I u\|_{L^2_x} \leq \|u\|_{L^2_x}$ which follows from the functional calculus and that $\|I(\xi)\|_{L^\infty_\xi} \leq 1$.
\end{rem}
\noindent \textit{Proof of \emph{Proposition} $\ref{prop almost conservation law}$.} By Lemma $\ref{lem local increment}$, we have for $N$ large enough,
\[
\sup_{t\in [0,\widetilde{T}_{\text{lwp}}]} |E(I u(t))-E(Iu_0)| \lesssim N^{-(2-\gamma+\delta)} \Big( \|Iu_0\|_{H^2_x}^{2+\frac{8}{d}} + \|Iu_0\|_{H^2_x}^{2+\frac{16}{d}} \Big).
\]
We only need to control $\|Iu_0\|_{H^2_x}$. To do so, we use the sharp Gagliardo-Nirenberg inequality $(\ref{sharp gargliardo nirenberg inequality})$ and $(\ref{property 1})$ to have
\begin{align*}
\|I u_0\|^2_{H^2_x} &\sim \|\Delta I u_0\|^2_{L^2_x} + \|I u_0\|_{L^2_x}^2 = 2E(Iu_0) + \frac{1}{1+\frac{4}{d}}\|I u_0\|^{2+\frac{8}{d}}_{L^{2+\frac{8}{d}}_x} +\|Iu_0\|_{L^2_x}^2 \\
&\leq 2E(Iu_0) + \Big(\frac{\|Iu_0\|_{L^2_x} }{\|Q\|_{L^2_x}}\Big)^{\frac{8}{d}} \|\Delta I u_0\|_{L^2_x}^2 + \|Iu_0\|^2_{L^2_x} \\
&\leq 2E(Iu_0) + \Big(\frac{\|u_0\|_{L^2_x} }{\|Q\|_{L^2_x}}\Big)^{\frac{8}{d}} \|I u_0\|_{H^2_x}^2 + \|u_0\|^2_{L^2_x}.
\end{align*}
Thus
\[
\Big(1 - \Big(\frac{\|u_0\|_{L^2_x} }{\|Q\|_{L^2_x}}\Big)^{\frac{8}{d}} \Big)\|Iu_0\|_{H^2_x}^2 \leq 2E(Iu_0) + \|u_0\|_{L^2_x}^2.
\]
By our assumptions $\|u_0\|_{L^2_x}<\|Q\|_{L^2_x}$ and $E(Iu_0)\leq 1$, we obtain $\|Iu_0\|_{H^2_x} \lesssim 1$. The proof is complete.
\defendproof
\section{Limiting profile} \label{section limiting profile}
\setcounter{equation}{0}
In this section, we prove Theorem $\ref{theorem weak limiting profile}$, Theorem $\ref{theorem mass concentration}$ and Theorem $\ref{theorem strongly limiting profile}$. 
\subsection{Proof of Theorem $\ref{theorem weak limiting profile}$} As the solution blows up at time $0<T^* <\infty$, the blowup alternative allows us to choose a sequence of times $(t_n)_{n\geq 1}$  such that $t_n \rightarrow T^*$ as $n\rightarrow \infty$ and $\|u(t_n)\|_{H^\gamma_x} =\Lambda(t_n)\rightarrow \infty$ as $n\rightarrow \infty$ (see $(\ref{energy increment notations})$ for the notation). Denote
\[
\psi_n(x):= \lambda_n^{\frac{d}{2}} I_{N(t_n)} u (t_n, \lambda_n x),
\]
where $N(t_n)$ is given as in $(\ref{define N_T})$ with $T=t_n$ and the parameter $\lambda_n$ is given by
\begin{align}
\lambda_n^2:=  \frac{\|\Delta Q\|_{L^2_x}}{\|\Delta I_{N(t_n)} u(t_n)\|_{L^2_x}}. \label{define lambda_n}
\end{align}
By $(\ref{property 4})$ and the blowup criterion given in Corollary $\ref{coro blowup criterion}$, we see that
\[
\lambda^2_n \lesssim \frac{\|\Delta Q\|_{L^2_x}}{\|u(t_n)\|_{H^\gamma_x}} \lesssim (T^*-t_n)^{\frac{\gamma}{4}} \text{ or } \lambda_n \lesssim (T^*-t_n)^{\frac{\gamma}{8}}.
\]
On the other hand, $(\psi_n)_{n\geq 1}$ is bounded in $H^2(\R^d)$. Indeed, 
\begin{align}
\|\psi_n\|_{L^2_x} &= \|I_{N(t_n)} u(t_n)\|_{L^2_x} \leq \|u(t_n)\|_{L^2_x} =\|u_0\|_{L^2_x}, \nonumber \\
\|\Delta \psi_n \|_{L^2_x} &=\lambda_n^2 \|\Delta I_{N(t_n)} u(t_n)\|_{L^2_x} = \|\Delta Q\|_{L^2_x}. \label{second derivative psi_n}
\end{align}
By Proposition $\ref{prop increment modified energy}$ with $T=t_n$, we have
\[
E(\psi_n) = \lambda_n^4 E(I_{N(t_n)}u(t_n))  \lesssim \lambda_n^4 \Lambda(t_n)^{a(\gamma)} \lesssim \Lambda(t_n)^{a(\gamma)-2}.
\]
As $0<a(\gamma)<2$ for $\frac{56-3d+\sqrt{137d^2+1712d+3136}}{2(2d+32)}<\gamma<2$, we see that $E(\psi_n)\rightarrow 0$ as $n\rightarrow \infty$. Therefore, the expression of the modified energy and $(\ref{second derivative psi_n})$ give
\begin{align}
\|\psi_n\|^{2+\frac{8}{d}}_{L^{2+\frac{8}{d}}_x} \rightarrow \Big(1+\frac{4}{d}\Big) \|\Delta Q\|_{L^2_x}^2, \label{strong convergence in L_2+8/d}
\end{align}
as $n\rightarrow \infty$. Applying Theorem $\ref{theorem concentration compactness}$ to the sequence $(\psi_n)_{n\geq 1}$ with $M=\|\Delta Q\|_{L^2_x}$ and $m=\left(\left(1+\frac{4}{d}\right)\|\Delta Q\|^2_{L^2_x}\right)^{\frac{d}{2d+8}}$, there exist a sequence $(x_n)_{n\geq 1} \subset \R^d$ and a function $U \in H^2(\R^d)$ such that $\|U\|_{L^2_x}\geq \|Q\|_{L^2_x}$ and up to a subsequence,
\[
\psi_n(\cdot + x_n) \rightharpoonup U \text{ weakly in } H^2(\R^d),
\]
as $n\rightarrow \infty$. That is 
\begin{align}
\lambda_n^{\frac{d}{2}} I_{N(t_n)} u(t_n, \lambda_n \cdot + x_n) \rightharpoonup  U \text{ weakly in } H^2(\R^d), \label{weak convergence H2}
\end{align}
as $n\rightarrow \infty$. To conclude Theorem $\ref{theorem weak limiting profile}$, we need to remove $I_{N(t_n)}$ from $(\ref{weak convergence H2})$. To do so, we consider for any $0 \leq \sigma <\gamma$,
\begin{align}
\|\lambda_n^{\frac{d}{2}} (u-I_{N(t_n)} u) (t_n, \lambda_n \cdot + x_n)\|_{\dot{H}^\sigma_x} & = \lambda_n^\sigma \|P_{\geq N(t_n)} u(t_n)\|_{\dot{H}^\sigma_x} \nonumber \\
&\lesssim \lambda_n^\sigma N(t_n)^{\sigma-\gamma} \|P_{\geq N(t_n)} u(t_n)\|_{\dot{H}^\gamma_x}  \nonumber \\ 
&\lesssim \Lambda(t_n)^{-\frac{\sigma}{2}} \Lambda(t_n)^{\frac{(\sigma-\gamma)a(\gamma)}{2(2-\gamma)}} \|P_{\geq N(t_n)} u(t_n)\|_{H^\gamma_x} \nonumber \\
&\lesssim \Lambda(t_n)^{1-\frac{\sigma}{2} +\frac{(\sigma-\gamma)a(\gamma)}{2(2-\gamma)}}. \label{exponent of Lambda t_n}
\end{align}
Using the explicit expression of $a(\gamma)$ given in $(\ref{define a gamma})$, we find that for
\[
\sigma <a(d,\gamma):=\frac{4d\gamma^2+(2d+48)\gamma+16d}{16d+(56-3d)\gamma-16\gamma^2},
\]
the exponent of $\Lambda(t_n)$ in $(\ref{exponent of Lambda t_n})$ is negative. Note that an easy computation shows that the condition $a(d,\gamma)<\gamma$ requires
\[
\frac{24-3d+\sqrt{9d^2+368d+576}}{32}<\gamma<2,
\]
which is satisfied by our assumption on $\gamma$. Thus,
\begin{align}
\|\lambda_n^{\frac{d}{2}} (u-I_{N(t_n)} u) (t_n, \lambda_n \cdot + x_n)\|_{H^{a(d,\gamma)-}_x} \rightarrow 0, \label{strong convergence H a(d,gamma)}
\end{align}
as $n\rightarrow \infty$. Combining $(\ref{weak convergence H2})$ and $(\ref{strong convergence H a(d,gamma)})$, we prove
\[
\lambda_n^{\frac{d}{2}} u(t_n, \lambda_n \cdot +x_n) \rightharpoonup U \text{ weakly in } H^{a(d,\gamma)-}(\R^d),
\]
as $n\rightarrow \infty$. The proof is complete.
\defendproof
\subsection{Proof of Theorem $\ref{theorem mass concentration}$} 
By Theorem $\ref{theorem weak limiting profile}$, there exists a blowup profile $U \in H^2(\R^d)$  with $\|U\|_{L^2_x} \geq \|Q\|_{L^2_x}$ and there exist sequences $(t_n, \lambda_n, x_n)_{n\geq 1} \subset \R_+\times \R^*_+ \times \R^d$ such that $t_n \rightarrow T^*$, 
\begin{align}
\frac{\lambda_n}{(T^*-t_n)^{\frac{\gamma}{8}}} \lesssim 1, \label{limit of lambda_n}
\end{align}
for all $n\geq 1$ and $\lambda_n^{\frac{d}{2}} u(t_n, \lambda_n \cdot + x_n) \rightharpoonup U$ weakly in $H^{a(d,\gamma)-}(\R^d)$ (hence in $L^2(\R^d)$) as $n\rightarrow \infty$. Thus for any $R>0$, we have
\[
\liminf_{n\rightarrow \infty} \lambda^d_n \int_{|x|\leq R} |u(t_n, \lambda_n x +x_n)|^2 dx \geq \int_{|x|\leq R}  |U(x)|^2 dx.
\]
By change of variables, we get
\[
\liminf_{n\rightarrow \infty} \sup_{y\in \R^d} \int_{|x-y|\leq R\lambda_n} |u(t_n, x)|^2 dx \geq \int_{|x|\leq R} |U(x)|^2 dx.
\]
Using the assumption $\frac{(T^*-t_n)^{\frac{\gamma}{8}}}{\alpha(t_n)} \rightarrow 0$ as $n\rightarrow \infty$, we have from $(\ref{limit of lambda_n})$ that $\frac{\lambda_n}{\alpha(t_n)} \rightarrow 0$ as $n\rightarrow \infty$. We thus obtain for any $R>0$, 
\[
\liminf_{n\rightarrow \infty} \sup_{y\in \R^d} \int_{|x-y|\leq \alpha(t_n)} |u(t_n,x)|^2 dx \geq \int_{|x|\leq R} |U(x)|^2 dx.
\]
Let $R\rightarrow \infty$, we obtain
\[
\liminf_{n\rightarrow \infty} \sup_{y\in\R^d} \int_{|x-y|\leq \alpha(t_n)} |u(t_n,x)|^2 dx \geq \|U\|_{L^2_x}^2.
\]
This implies 
\[
\limsup_{t\nearrow T^*} \sup_{y\in \R^d} \int_{|x-y|\leq \alpha(t)} |u(t,x)|^2 dx \geq \|Q\|^2_{L^2_x}. 
\]
Sine for any fixed time $t$, the map $y\mapsto \int_{|x-y|\leq \alpha(t)} |u(t,x)|^2 dx$ is continuous and goes to zero as $|y|\rightarrow \infty$, there exists $x(t) \in \R^d$ such that
\[
\sup_{y\in \R^d}  \int_{|x-y|\leq \alpha(t)} |u(t,x)|^2 dx = \int_{|x-x(t)|\leq \alpha(t)} |u(t,x)|^2 dx.
\]
This shows
\[
\limsup_{t \nearrow T^*} \int_{|x-x(t)| \leq \alpha(t)} |u(t,x)|^2 dx \geq \|Q\|^2_{L^2_x}.
\]
The proof is complete.
\defendproof
\subsection{Proof of Theorem $\ref{theorem strongly limiting profile}$} We firstly recall the following variational characterization of the solution to the ground state equation $(\ref{ground state equation})$. Note that the uniqueness up to translations in space, phase and dilations of solution to this ground state equation is assumed here. 
\begin{lem}[Variation characterization of the ground state \cite{ZhuYangZhang10}] \label{lem variational characterization} If $v \in H^2(\R^d)$ is such that $\|v\|_{L^2_x}= \|Q\|_{L^2_x}$ and $E(u)=0$, then $v$ is of the form
\[
v(x)= e^{i\theta} \lambda^{\frac{d}{2}} Q(\lambda x+ x_0),
\] 
for some $\theta \in \R, \lambda>0$ and $x_0\in \R^d$, where $Q$ is the unique solution to the ground state equation $(\ref{ground state equation})$.
\end{lem}  
Using the notation in the proof of Theorem $\ref{theorem weak limiting profile}$ and the assumption $\|u_0\|_{L^2_x} = \|Q\|_{L^2_x}$, we have
\[
\|\psi_n\|_{L^2_x} \leq \|u_0\|_{L^2_x} = \|Q\|_{L^2_x} \leq \|U\|_{L^2_x}. 
\]
Sine $\psi_n(\cdot +x_n) \rightharpoonup U$ weakly in $L^2(\R^d)$, the semi-continuity of weak convergence implies
\[
\|U\|_{L^2_x} \leq \liminf_{n\rightarrow \infty} \|\psi_n\|_{L^2_x} \leq \|Q\|_{L^2_x}.
\]
Thus, 
\begin{align}
\|U\|_{L^2_x} = \|Q\|_{L^2_x} = \lim_{n\rightarrow \infty} \|\psi_n\|_{L^2_x}. \label{L2 norm U and Q}
\end{align}
Hence up to a subsequence 
\begin{align}
\psi_n(\cdot +x_n) \rightarrow U \text{ strongly in } L^2(\R^d), \label{strongly convergence in L2}
\end{align} 
as $n\rightarrow \infty$. On the other hand, using $(\ref{second derivative psi_n})$, the Gagliardo-Nirenberg inequality $(\ref{sharp gargliardo nirenberg inequality})$ implies $\psi_n(\cdot +x_n) \rightarrow U$ strongly in $L^{2+\frac{8}{d}}(\R^d)$. Indeed, by $(\ref{second derivative psi_n})$,
\begin{align*}
\|\psi_n(\cdot+x_n) - U\|_{L^{2+\frac{8}{d}}_x}^{2+\frac{8}{d}} &\lesssim \|\psi(\cdot +x_n) -U\|^{\frac{8}{d}}_{L^2_x} \|\Delta(\psi_n(\cdot +x_n) - U\|_{L^2_x}^2 \\
&\lesssim (\|\Delta Q\|_{L^2_x} + \|\Delta U\|_{L^2_x})^2 \|\psi(\cdot +x_n) -U\|^{\frac{8}{d}}_{L^2_x} \rightarrow 0,
\end{align*}
as $n\rightarrow \infty$. Moreover, using $(\ref{strong convergence in L_2+8/d})$ and $(\ref{L2 norm U and Q})$, the sharp Gagliardo-Nirenberg inequality $(\ref{sharp gargliardo nirenberg inequality})$ also gives
\[
\|\Delta Q\|^2_{L^2_x} = \frac{1}{1+\frac{4}{d}} \|U\|^{2+\frac{8}{d}}_{L^{2+\frac{8}{d}}_x} \leq \Big(\frac{\|U\|_{L^2_x}}{\|Q\|_{L^2_x}}\Big)^{\frac{8}{d}} \|\Delta U\|_{L^2_x}^2 = \|\Delta U\|_{L^2_x}^2,
\]
or $\|\Delta Q\|_{L^2_x} \leq \|\Delta U\|_{L^2_x}$. By the semi-continuity of weak convergence and $(\ref{second derivative psi_n})$,
\[
\|\Delta U\|_{L^2_x} \leq \liminf_{n\rightarrow \infty} \|\Delta \psi_n\|_{L^2_x} = \|\Delta Q\|_{L^2_x}.
\]
Therefore,
\begin{align}
\|\Delta U\|_{L^2_x}=\|\Delta Q\|_{L^2_x} = \lim_{n\rightarrow \infty} \|\Delta \psi_n\|_{L^2_x}. \label{H dot 2 norm U and Q}
\end{align}
Combining $(\ref{L2 norm U and Q}), (\ref{H dot 2 norm U and Q})$ and using the fact $\psi_n(\cdot +x_n)\rightharpoonup U$ weakly in $H^2(\R^d)$, we conclude that $\psi_n(\cdot +x_n) \rightarrow U$ strongly in $H^2(\R^d)$. In particular,
\[
E(U) = \lim_{n\rightarrow \infty} E(\psi_n) = 0,
\]
as $n\rightarrow \infty$. This shows that there exists $U \in H^2(\R^d)$ satisfying
\[
\|U\|_{L^2_x}=\|Q\|_{L^2_x}, \quad \|\Delta U\|_{L^2_x} = \|\Delta Q\|_{L^2_x}, \quad E(U)=0.
\]
Applying the variational characterization given in Lemma $\ref{lem variational characterization}$, we have (taking $\lambda=1$),
\[
U(x)=e^{i\theta} Q(x+x_0),
\]
for some $(\theta, x_0) \in \R \times \R^d$. Hence
\[
\lambda_n^{\frac{d}{2}} I_{N(t_n)} u(t_n, \lambda_n \cdot +x_n) \rightarrow e^{i\theta} Q(\cdot + x_0) \text{ strongly in } H^2(\R^d),
\]
as $n\rightarrow \infty$. Using $(\ref{strong convergence H a(d,gamma)})$, we prove
\[
\lambda_n^{\frac{d}{2}} u(t_n, \lambda_n\cdot +x_n) \rightarrow e^{i\theta} Q(\cdot +x_0) \text{ strongly in } H^{a(d,\gamma)-}(\R^d),
\]
as $n\rightarrow \infty$. The proof is complete.
\defendproof
\section{Global well-posedness} \label{section global well-posedness}
\setcounter{equation}{0}
In this section, we will give the proof of Theorem $\ref{theorem global existence below ground state}$. By density argument, we assume that $u_0 \in C^\infty_0(\R^d)$. Let $u$ be a global solution to (NL4S) with initial data $u_0$ satisfying $\|u_0\|_{L^2_x} <\|Q\|_{L^2_x}$. In order to apply the almost conservation law given in Proposition $\ref{prop almost conservation law}$, we need the absolute value of modified energy of initial data is small. Since $E(Iu_0)$ is not necessarily small, we will use the scaling $(\ref{scaling invariance})$ to make $E(Iu_\lambda(0))$ is small. We have
\begin{align*}
E(Iu_\lambda(0)) \leq \frac{1}{2}\|\Delta I u_\lambda(0)\|^2_{L^2_x} \lesssim N^{2(2-\gamma)} \|\Delta u_\lambda(0)\|^2_{L^2_x} = N^{2(2-\gamma)} \lambda^{-2\gamma} \|u_0\|^2_{\dot{H}^2_x}. 
\end{align*}
Thus, we can make $E(Iu_\lambda(0)) \leq \frac{1}{4}$ by taking 
\begin{align}
N\sim \lambda^{\frac{2-\gamma}{\gamma}}. \label{choice of N}
\end{align} 
Moreover, since the scaling $(\ref{scaling invariance})$ preserves the $L^2$-norm, we have $\|u_\lambda(0)\|_{L^2_x} = \|u_0\|_{L^2_x}<\|Q\|_{L^2_x}$. Thus, the assumptions of Proposition $\ref{prop almost conservation law}$ are satisfied. Therefore, there exists $\tau>0$ so that for $N$ sufficiently large,
\[
E(Iu_\lambda(t)) \leq E(Iu_\lambda(0)) + C N^{-(2-\gamma+\delta)}, 
\] 
for $t \in [0,\tau]$ where $\max\{3-\frac{8}{d}, \frac{8}{d}\}<\gamma<2$ and $0<\delta<\gamma+\frac{8}{d}-3$. We may reapply this proposition continuously so that $E(Iu_\lambda(t))$ reaches 1, that is at least $C_1N^{2-\gamma+\delta}$ times. Therefore,
\begin{align}
E(I u_\lambda(C_1 \tau N^{2-\gamma+\delta})) \sim 1. \label{modified energy iteration}
\end{align}
Now, given any $T\gg 1$, we choose $N\gg 1$ so that
\[
T \sim C_1 \tau\frac{N^{2-\gamma+\delta}}{\lambda^4}.
\] 
Using $(\ref{choice of N})$, we have
\begin{align}
T\sim N^{2-\gamma+\delta-\frac{4(2-\gamma)}{\gamma}}. \label{choice of T}
\end{align}
As $0<\delta<\gamma+\frac{8}{d}-3$ or $2-\gamma+\delta <\frac{8}{d}-1$, the exponent of $N$ is positive provided that
\[
\frac{8}{d}-1 - \frac{4(2-\gamma)}{\gamma} >0 \text{ or } \gamma>\frac{8d}{3d+8}.
\] 
Thus the choice of $N$ makes sense for arbitrary $T\gg 1$. A direct computation and $(\ref{choice of N}), (\ref{modified energy iteration})$ and $(\ref{choice of T})$ show 
\begin{align*}
E(I u(T)) = \lambda^4 E(Iu_\lambda(\lambda^4T)) = \lambda^4E(I u_\lambda(C_1 \tau N^{2-\gamma+\delta}) \sim \lambda^4 \leq N^{\frac{4(2-\gamma)}{\gamma}} \sim T^{\frac{4(2-\gamma)}{(2-\gamma+\delta)\gamma-4(2-\gamma)}}.
\end{align*}
This shows that there exists $C_2=C_2(\tau, \|u_0\|_{H^\gamma_x})$ such that
\[
E(Iu(T)) \leq C_2 T^{\frac{4(2-\gamma)}{(2-\gamma+\delta)\gamma-4(2-\gamma)}},
\]
for any $T\gg 1$. Finally, by $(\ref{property 4})$,
\begin{align*}
\|u(T)\|^2_{H^\gamma_x} &\lesssim \|I u(T)\|^2_{H^2_x} \sim \|\Delta I u(T)\|^2_{L^2_x} + \|Iu(T)\|^2_{L^2_x} \lesssim E(Iu (T)) + \|u_0\|_{L^2_x}^2 \\
&\lesssim C_3 T^{\frac{4(2-\gamma)}{(2-\gamma+\delta)\gamma-4(2-\gamma)}} + C_4,
\end{align*}
where $C_3, C_4$ depends only on $\|u_0\|_{H^\gamma_x}$. The proof is complete.
\defendproof
\section*{Acknowledgments}
The author would like to express his deep gratitude to Prof. Jean-Marc BOUCLET for the kind guidance and encouragement. 


\end{document}